\RequirePackage{fix-cm}
\documentclass[smallextended]{svjour3}       
\smartqed  
\usepackage{times,amsmath,amsbsy,amssymb}
\usepackage{graphicx,color,float,tikz}
\begin{document}

\title{A finite element method for Dirichlet boundary control of elliptic partial differential equations}\thanks{This work is supported in part by the Program of 
Chongqing Innovation Team Project in University under Grant CXTDX201601022.}


\titlerunning{A finite element method for Dirichlet boundary control problem}     

\author{Shaohong Du          \and         Zhiqiang Cai}

\authorrunning{S. H. Du, Z. Q. Cai} 

\institute{Shaohong Du\at
               School of Mathematics and Statistics, Chongqing Jiaotong University, Chongqing 400074, China\\
              Tel.: +86-15923295341\\
              Fax: +86-023-62652579\\
              \email{dushaohong@csrc.ac.cn}           
           \and
           Zhiqiang Cai \at
           Department of Mathematics, Purdue University, 150 N. University Street, West Lafayette, IN 47907-2067, USA\\
           Tel.: +1-925-640-5055\\
            Fax: +1-765-494-0548\\
            \email{caiz@purdue.com}
              }
 \date{Received: date / Accepted: date}

\maketitle

\begin{abstract}
This paper introduces a new variational formulation for Dirichlet boundary control problem of elliptic partial differential equations, based on observations that the state and adjoint state are related through 
the control on the boundary of the domain, and that such a relation may be imposed in the variational formulation of the adjoint state. Well-posedness (unique solvability and stability) of the variational problem is
established in the $H^{1}(\Omega)\times H_{0}^{1}(\Omega)$ space for the respective state and adjoint state. A finite element method based on
this formulation is analyzed. It is shown that the conforming $k-$th order finite element approximations to the state and the adjoint state,
in the respective $L^{2}$ and $H^{1}$ norms converge at the rate of order $k-1/2$ on quasi-uniform mesh for conforming element of order $k$. Numerical examples are presented to validate the theory.
\end{abstract}
\maketitle

\section {Introduction}\label{SECTION1}
Let $\Omega\subset\mathbb{R}^{d},d\geq2$, be a bounded polygonal or polyhedral domain with Lipshitz boundary $\Gamma=\partial\Omega$. Consider
the following Dirichlet boundary control problem of elliptic partial differential equations (PDEs):
\begin{equation}\label{du1}
\min\ J(u),\ \ J(u)=\displaystyle\frac{1}{2}||y-y_{d}||_{L^{2}(\Omega)}^{2}+\frac{\gamma}{2}||u||_{L^{2}(\Gamma)}^{2},
\end{equation}
where the regularization parameter $\gamma>0$ and $y$ is the solution of the Poisson equation with nonhomogeneous Dirichlet boundary conditions
\begin{eqnarray}
    \label{du2}
    -\triangle y=f\ \ {\rm in}\ \Omega, \\
    \label{du3}
    y=u\ \ {\rm on}\ \Gamma.
\end{eqnarray}

 After the pioneering works of Falk \cite{Falk} and Geveci \cite{Geveci}, there were some efforts on the  error estimates for finite element approximation to control problems governed by PDEs. Arada et al. in \cite{Arada,Casas2002} derived error estimates for the control in the $L^{\infty}$ and $L^{2}$ norms for semilinear elliptic control problem. The articles \cite{Gunzburger1996,Fursikov} studied the error estimates of finite element approximation for some important flow control problems. Casas et al. \cite{Casas2005} carried out the study of the Neumann boundary control problem.

However, the works mentioned above are mainly contributions to the distributed control. Since the Dirichlet boundary control plays an important
role in many applications such as flow control problems and has been a hot topic for decades. It is well known that Dirichlet control problems
are difficult theoretically and numerically, because the Dirichlet
boundary data does not directly enter a standard variational setting for the PDEs. On the one hand, the traditional finite element method (see, e.g., \cite{Casa2006,Vexler2007,Deckelnick2009,May2013,Apel2017}) deals with the state variable ($y$) using its weak formulation, e.g., allowing
for solutions $y\in L^{2}(\Omega)$; on the other hand, the attempt of the first order optimality condition involves the normal derivative of the adjoint state ($z$) on the boundary of the domain. Therefore, it is crucial to obtain this normal derivative numerically by using additional equation. But in doing so the problem becomes complicated in both theoretical analysis and numerical practice. Note that the regularity of the solution and error estimates for finite element approximates have been studied in \cite{Apel2015,Mateos2016,Apel2017}.

To avoid the difficulty described above, there are two ways to deal with the control variable. One is in \cite {Of2015} replaced the $L^{2}$ norm in the cost functional with the $H^{1/2}$ norm and attained {\it a priori} estimate of the numerical error of the control by using piecewise linear elements, and the other approximate the nonhomogeneous Dirichlet boundary condition with a Robin boundary condition or weak boundary penalization. However, the former changed the problem and the latter had to deal with the penalization which is computationally expensive. Recently, techniques similar to \cite {Of2015} were used in \cite{Gunzburger1991,Gunzburger1992,Chowdhury2017}.

Recently, Gong et al. considered the mixed finite element method in \cite{Gong2011}, where the optimal control and the adjoint state were involved in a variational form in a natural sense. This makes its theoretical analysis easier, but the corresponding fluxes of the two states $y$ and $z$ are required to be introduced. It points out that the mixed finite method obtained the same rate of convergence as the regularity of the control on boundary. Apel et al. \cite{Apel2017} have considered a standard finite element method on a special class of meshes and guaranteed a superlinear convergence rate for the control. Very recently, Hu et al. \cite{Hu2017} considered a hybridizable discontinuous Galerkin method to obtain optimal {\it a priori} error estimates for the control by solving an algebraic system of seven unknown functions.

Based on both the facts that the control $u$ is equal to the restriction of the state $y$ on the boundary (see the original equation (\ref{du3})), and that the restriction of an approximation of the state $y$ on the boundary is also an approximation to the control $u$, we realize that the restriction of the numerical error for the state $y$ on the boundary can be used to measure the numerical error of the control $u$ in the $L^{2}(\Gamma)$-norm. This idea is done in the way that the state $y$ and its adjoint state $z$ will be coupled by the original equation (\ref{du3}) and an extra equation (\ref{du8}) as well as by the right-hand side term $y$ of the equation (\ref{du6}), i.e., the control $u$ and the
normal derivative of the adjoint state $z$ along the boundary are cancelled. This is different from the idea in literatures e.g., \cite{Casa2006,Vexler2007,Deckelnick2009,May2013,Apel2017,Of2015,Gunzburger1991,Gunzburger1992,Chowdhury2017}, where both the original
equation (\ref{du3}) and an extra equation (\ref{du8}) were taken into account in variational formulation.

This paper introduces a new variational formulation for Dirichlet boundary control problem of elliptic partial differential equations, based on observations that the state and adjoint state are related through the control on the boundary of the domain, and that such a relation may be imposed in the variational formulation of the adjoint state, i.e., one can substitute the control by the control law (the control is the normal derivative of the adjoint state (up to a factor)). Well-posedness (unique solvability and stability) of the variational problem is
established in the $H^{1}(\Omega)\times H_{0}^{1}(\Omega)$ space for the respective state and adjoint state. A finite element method based on
this formulation is analyzed. It is shown that the conforming $k-$th order finite element approximations to the state and the adjoint state,
in the respective $L^{2}$ and $H^{1}$ norms converge at the rate of order $k-1/2$ on quasi-uniform mesh for conforming element of order $k$. Numerical examples are presented to validate the theory.

This paper is organized as follows. In Section~\ref{SECTION2}, we introduce a new variational setting based on an observation. Section~\ref{SECTION3} is devoted to the unique solvability and stability of the variational problem. In Section~\ref{SECTION4}, we introduce finite element approximation to the variational setting and prove a preliminary result, which will prepare us for the {\it a priori} error estimation on an approximation of the conforming element of order $k$ over quasi-uniform mesh in Section~\ref{SECTION5}. In Section \ref{SECTION6A}, we analyze the stability of the discrete control in $L^{2}(\Gamma)$ norm and $H^{1/2}(\Gamma)$ norm in the sense that the restriction of the discrete state on the boundary is considered as an approximation of the control.  Finally numerical tests are provided in Section~\ref{SECTION6} to support our theory.

\section{A variational formulation}\label{SECTION2}
For any bounded open subset $\omega$ of $\Omega$ with Lipschitz
boundary $\gamma$, let $L^{2}(\gamma)$ and $H^{m}(\omega)$ be the standard Lebesgue and Sobolev spaces equipped with standard norms
$\|\cdot\|_{\gamma}=\|\cdot\|_{L^{2}(\gamma)}$ and $\|\cdot\|_{m,\omega}=\|\cdot\|_{H^{m}(\omega)}$,
$m\in\mathbb{N}$. Note that $H^{0}(\omega)=L^{2}(\omega)$. Denote by $|\cdot|_{m,\omega}$ the
semi-norm in $H^{m}(\omega)$. Similarly, denote by $(\cdot,\cdot)_{\gamma}$ and
$(\cdot,\cdot)_{\omega}$ the $L^{2}$ inner products on $\gamma$
and $\omega$, respectively. We shall omit the symbol $\Omega$ in the notations above if
$\omega=\Omega$.

It is well known that the Dirichlet boundary control problem in (\ref{du1})-(\ref{du3}) is equivalent to the optimality system
\begin{eqnarray}
    \label{du4}
    -\triangle y=f\ \ \ \ \ {\rm in}\ \ \ \ \Omega, \\
    \label{du5}
    y=u\ \ \ \ \ {\rm on}\ \ \ \ \Gamma, \\
    \label{du6}
    -\triangle z=y-y_{d}\ \ \ \ \ {\rm in}\ \ \ \ \Omega, \\
    \label{du7}
    z=0\ \ \ \ \ {\rm on}\ \ \ \ \Gamma, \\
    \label{du8}
    u=\frac{1}{\gamma}\frac{\partial z}{\partial{\bf n}}\ \ \ \ \ \ {\rm on}\ \ \ \ \Gamma,
\end{eqnarray}
where ${\bf n}$ is the unit outer normal to $\Gamma$. Note that these equations must be understood in a weak sense.

To see the idea of variational setting, we consider the following several cases under an assumption that the domain and known data are
respectively satisfied with these cases:

Case one: the state $y\in H^{1/2}(\Omega)$, so $y|_{\Gamma}$ belongs to $L^{2}(\Gamma)$. We know $u\in L^{2}(\Gamma)$ from (\ref{du5}). The equation (\ref{du8}) implies that $\partial z/\partial {\bf n}|_{\Gamma}\in L^{2}(\Gamma)$, this needs the adjoint state to
satisfy $z\in H^{3/2}(\Omega)$.

Case two: $y\in H^{1}(\Omega),y|_{\Gamma}=u\in H^{1/2}(\Gamma)$, the equation (\ref{du8}) means that $\partial z/\partial {\bf n}|_{\Gamma}\in H^{1/2}(\Gamma)$, this requires $z\in H^{2}(\Omega)$.

Case three: $y\in L^{2}(\Omega),y|_{\Gamma}=u\in H^{-1/2}(\Gamma)$ (the dual space of $H^{1/2}(\Gamma)$), the equation (\ref{du8}) requires $\partial z/\partial {\bf n}|_{\Gamma}\in H^{-1/2}(\Gamma)$, this indicates $z\in H^{1}(\Omega)$.

Case four: $y\in H^{3/2}(\Omega),y|_{\Gamma}=u\in H^{1}(\Gamma)$, the equation (\ref{du8}) needs $\partial z/\partial {\bf n}|_{\Gamma}\in H^{1}(\Gamma)$, this requires $z\in H^{5/2}(\Omega)$.

Since natural functional analytical setting of this problem uses $L^{2}(\Gamma)$ as a ``control space", Case one is an ideal choice for the control
$u$, state $y$, and adjoint state $z$. However, it is difficult to bring this characteristics of $y$ and $z$ into their respective variational
formulation if (\ref{du5}) and (\ref{du8}) are regarded as two independent equations. For Cases two and three, it is convenient to incorporate the spaces of $y$ and $z$ into their respective variational formulation, but doing so expands or narrows down the space of the control $u$, and can not provide the variational formulation of $u$ if (\ref{du3}) and (\ref{du8}) are still regarded as two independent equations. Case four further does not only enlarges the space of $u$, but also requires a higher regularity on $y$ and $z$, and bring an unexpected difficulty to variation and computation.

These cases show that it is difficult to keep the compatibility of the spaces of $u,y$ and $z$ and incorporate their respective space into their respective variational formulation. We realize that the state $y$ and adjoint state $z$ are connected by the control $u$ on the boundary (see (\ref{du5}) and (\ref{du8})) as well as by the state $y$ being the right-hand side term of the equation of the adjoint state (see (\ref{du6})), and that it can overcome these difficulties referred above to cancel the control $u$ and to absorb $\displaystyle\frac{1}{\gamma}\frac{\partial z}{\partial {\bf n}}\big|_{\Gamma}=y|_{\Gamma}$ into the variational formulation of $z$ as a boundary condition.

Based on the above observation, the Dirichlet boundary condition in (\ref{du5}) indicates that the control $u$ is equal to the restriction of the state $y$ on the boundary $\Gamma$. Therefore, we simultaneously obtain the control $u$ if the state $y$ is got. The equation (\ref{du8}) is an
additional equation with respect to the adjoint state $z$. Here, we do not regard (\ref{du8}) as an additional equation, instead
we understand it as a boundary condition, through which the state $y$ and its adjoint state $z$ will be coupled. So the control
$u$ can be cancelled in form, but it can be reflected by the state $y$ in essence. It is pointed out that the right hand term of
(\ref{du6}) includes the state variable $y$, through which the adjoint state $z$ is coupled over the whole domain.

Based on this idea, multiplying both sides of (\ref{du4}) by $\psi\in H_{0}^{1}(\Omega)$, and applying integration by parts, we attain
\begin{equation}\label{du8a}
\displaystyle\int_{\Omega}\nabla y\cdot\nabla\psi d{\bf x}=\int_{\Omega}f\psi d{\bf x}.
\end{equation}
Similarly, multiplying both sides of (\ref{du6}) by $\phi\in H^{1}(\Omega)$, an integration by parts leads to
\begin{equation}\label{du8b}
\displaystyle\int_{\Omega}\nabla z\cdot\nabla\phi d{\bf x}-\int_{\Gamma}\frac{\partial z}{\partial{\bf n}}\phi ds=\int_{\Omega}(y-y_{d})\phi d{\bf x}.
\end{equation}
Cancelling $u$ from a combination of (\ref{du5}) and (\ref{du8}) yields to
\begin{equation}\label{du8c}
\displaystyle\frac{\partial z}{\partial{\bf n}}\big|_{\Gamma}=\gamma y|_{\Gamma}.
\end{equation}
Substituting (\ref{du8c}) into (\ref{du8b}), we get
\begin{equation}\label{du8d}
\displaystyle\int_{\Omega}\nabla z\cdot\nabla\phi d{\bf x}-\int_{\Gamma}\gamma y\phi ds=\int_{\Omega}(y-y_{d})\phi d{\bf x}.
\end{equation}

Collecting (\ref{du8a}) and (\ref{du8d}) gives the following variational formulation: Find $(y,z)\in H^{1}(\Omega)\times H_{0}^{1}(\Omega)$ such that
\begin{eqnarray}
    \label{du9}
    (\nabla y,\nabla\psi)=(f,\psi)\ \ \ \ \forall\ \psi\in H_{0}^{1}(\Omega), \\
    \label{du10}
    (\nabla z,\nabla\phi)-(\gamma y,\phi)_{\Gamma}-(y,\phi)=-(y_{d},\phi)\ \ \ \ \forall\ \phi\in H^{1}(\Omega).
\end{eqnarray}

In what follows, we clarify the unique solvability of the variational problem in (\ref{du9})-(\ref{du10}). For a 2D convex polygonal domain, we recall a regularity result of May et al. in \cite{May2013} below, which gives conditions on the domain and data to guarantee the regularity of the solution. To this end,
let $\omega_{\rm max}$ be the maximum interior angle of the polygonal domain $\Omega$, and denote $p_{*}^{\Omega}$ by
\begin{equation}\label{Febsa}
p_{*}^{\Omega}=2\omega_{\rm max}/(2\omega_{\rm max}-\pi),
\end{equation}
including the special case $p_{*}^{\Omega}=\infty$ for $\omega_{\rm max}=\pi/2$. For a higher dimensional convex polygonal domain, we
do not attempt to provide condition on the regularity of the solution, because we put an emphasis on a variational setting and the corresponding finite element approximation. Of course, the regularity theory is more complicated in three-dimensional case.

\begin{lemma}\label{Febsaa}
{\em(\cite{May2013} Lemma 2.9).} Suppose that $f\in L^{2}(\Omega)$ and $y_{d}\in L^{p_{*}^{d}}(\Omega),p_{*}^{d}>2$, and
that $\Omega\subset\mathbb{R}^{2}$ is a bounded convex domain with polygonal boundary $\Gamma$. Let $p_{*}^{\Omega}\geq2$
be defined by {\em(\ref{Febsa})} and $p_{*}: =\min(p_{*}^{d},p_{*}^{\Omega})$. Then, the solution $(y,u)$ of the optimization
problem {\em(\ref{du1})-(\ref{du3})} and the associated adjoint state determined by {\em(\ref{du6})} have the regularity properties
\begin{equation*}
(y,u,z)\in H^{3/2-1/p}(\Omega)\times H^{1-1/p}(\Gamma)\times(H_{0}^{1}(\Omega)\cap W_{p}^{2}(\Omega)),\ \ 2\leq p<p_{*}.
\end{equation*}
\end{lemma}

Owing to the optimal system (\ref{du4})-(\ref{du8}) equivalent to the problem (\ref{du1})-(\ref{du3}), the regularity of the
solution for the system (\ref{du9})-(\ref{du10}) is guaranteed in terms of Lemma \ref{Febsaa} in case of $p=2$.

\section{Unique solvability and stability}\label{SECTION3}
This section establishes unique solvability for the variational problem in (\ref{du9})-(\ref{du10}), and stability estimate of the control and the state and the adjoint state variables.
\begin{theorem}
For $f\in H^{-1}(\Omega)$ and $y_{d}\in L^{2}(\Omega)$, the system {\em(\ref{du9})-(\ref{du10})} is uniquely solvable, and is stable in the sense that there exists a positive constant
$C_{\gamma}$, depending on $\gamma$, such that
\begin{equation}\label{dudu89}
\gamma^{1/2}||u||_{0,\Gamma}+||y||+||\nabla z||\leq C_{\gamma}\left(||f||_{-1}+||y_{d}||\right).
\end{equation}
\end{theorem}
\begin{proof}
We first prove that the variational problem in (\ref{du9})-(\ref{du10}) is solvable. Since the existence of the solution for the
optimization problem in (\ref{du1})-(\ref{du3}) has been proven by introducing the so-called ``solution operator" and using convex
analysis (see Lemma 2.4 in \cite{May2013}), and the first order optimal condition shows that the solution of the optimization problem in (\ref{du1})-(\ref{du3}) satisfies (\ref{du4})-(\ref{du8}). Hence, the system (\ref{du4})-(\ref{du8}) is solvable. Due to the solution of (\ref{du4})-(\ref{du8}) satisfying the variational problem in (\ref{du9})-(\ref{du10}), then (\ref{du9})-(\ref{du10})
 has a solution.

In what follows, we prove the stability of the system in (\ref{du9})-(\ref{du10}). By (\ref{du10}) with $\psi=z$, we obtain
\begin{equation*}
\begin{array}{lll}
||\nabla z||^{2}&=&(\gamma y,z)_{\Gamma}+(y,z)-(y_{d},z)\vspace{2mm}\\
&\leq&\gamma||y||_{-1/2,\Gamma}||z||_{1/2,\Gamma}+||y||_{-1}||z||_{1}+||y_{d}||_{-1}||z||_{1}\vspace{2mm}\\
&\leq& C(\gamma||y||_{0,\Gamma}||z||_{1}+\|y\|\,\|z\|_{1}+\|y_{d}\|\,\|z\|_{1}),
\end{array}
\end{equation*}
which, together with the Poinc\'{a}re inequality, implies
\begin{equation}\label{dudu91}
||z||_{1}\leq C\left(\gamma||y||_{0,\Gamma}+||y||+||y_{d}||\right).
\end{equation}

It follows from (\ref{du10}) with $\psi=y$, (\ref{du9}), the Cauchy-Schwarz and Young inequalities, and (\ref{dudu91}) that
\begin{equation*}
\begin{array}{lll}
\gamma||y||_{0,\Gamma}^{2}+||y||^{2}&=&(\nabla z,\nabla y)+(y_{d},y)\vspace{2mm}\\
&=&(f,z)+(y_{d},y)\leq||f||_{-1}||z||_{1}+\|y_{d}\|\,\|y\|\vspace{2mm}\\
&\leq&C||f||_{-1}\left(\gamma||y||_{0,\Gamma}+||y||+||y_{d}||\right)+\|y_{d}\|\,\|y\|\vspace{2mm}\\
&\leq&C\left(||f||_{-1}+||y_{d}||\right)\left(\gamma||y||_{0,\Gamma}+||y||\right)+C\left(||f||_{-1}^{2}+||y_{d}||^{2}\right),
\end{array}
\end{equation*}
which implies
\begin{equation*}
\gamma||y||_{0,\Gamma}^{2}+||y||^{2}\leq C_{\gamma}\left(||f||_{-1}^{2}+||y_{d}||^{2}\right).
\end{equation*}

Now, (\ref{dudu89}) is a direct consequence of (\ref{dudu91}) and the fact that $u=y$ on $\Gamma$, and the uniqueness of the solution
follows from (\ref{dudu89}) immediately, since the corresponding homogeneous system has vanishing solution. This completes the proof of the
theorem.
\end{proof}

\begin{theorem}\label{dudu93}
Assume that the domain $\Omega$ is convex with Lipshitz boundary. For $f\in H^{-1}(\Omega)$ and $y_{d}\in L^{2}(\Omega)$, there exists
a positive constant $C_{\gamma}$ dependent on $\gamma$ such that
\begin{equation}\label{Dec1}
||\nabla y||\leq C_{\gamma}\left(||f||_{-1}+||y_{d}||\right).
\end{equation}
\end{theorem}
\begin{proof}
By the standard $H^{1}(\Omega)$ a priori estimate of the problem in (\ref{du4})-(\ref{du5}), and equation (\ref{du8}), we have
\begin{equation}\label{Dec4}
\begin{array}{lll}
||\nabla y||&\leq& C\left(||f|_{-1}+||u||_{1/2,\Gamma}\right)\vspace{2mm}\\
&\leq&\displaystyle C\left(||f|_{-1}+\frac{1}{\gamma}\Big\|\frac{\partial z}{\partial {\bf n}}\Big\|_{1/2,\Gamma}\right).
\end{array}
\end{equation}
The standard $H^{2}(\Omega)$ a priori estimate (see, e.g., \cite{May2013,Casas2009}) of the problem in (\ref{du5})-(\ref{du6}) gives
\begin{equation}\label{Dec5}
\begin{array}{lll}
\displaystyle||z||_{2}+\Big\|\frac{\partial z}{\partial{\bf n}}\Big\|_{1/2,\Gamma}&\leq& C||y-y_{d}||\vspace{2mm}\\
&\leq& C\left(||y||+||y_{d}||\right).
\end{array}
\end{equation}

Now, (\ref{Dec1}) is a direct consequence of (\ref{Dec4}), (\ref{Dec5}), and (\ref{dudu89}).
\end{proof}

\begin{remark}
Due to $z\in H_{0}^{1}(\Omega)$, the Pioncar\'{e} inequality implies $||z||_{1}\leq C||\nabla z||$. Hence, under the
assumption of Theorem \ref{dudu93}, it holds the following stable estimate:
\begin{equation*}
\gamma^{1/2}||u||_{0,\Gamma}+||y||_{1}+||z||_{1}\leq C_{\gamma}\left(||f||_{-1}+||y_{d}|\right).
\end{equation*}
\end{remark}

\section{Finite element approximation and preliminary result}\label{SECTION4}
We introduce the discrete formulation of (\ref{du9})-(\ref{du10}). To this end, let $\mathcal{T}_{h}$ be a
partition of $\Omega$ into triangles (tetrahedra for $d=3$) or parallelograms (parallelepiped
for $d=3$). With each element $K\in\mathcal{T}_{h}$, we associate two parameters $\rho(K)$ and $\sigma(K)$,
where $\rho(K)$ denotes the diameter of the set $K$, and $\sigma(K)$ is the diameter of the largest ball
contained in $K$. Let us define the size of the mesh by $h=\max_{K\in\mathcal{T}_{h}}\rho(K)$.
About the partition, we also assume that there exists a constant $\rho>0$
such that $h/\rho(K)\leq\rho$ for all $K\in\mathcal{T}_{h}$ and $h>0$, i.e., the mesh $\mathcal{T}_{h}$ is quasi-uniform.

Denote $P_{k}(K)$ be the space of polynomials of total degree at most $k$ if $K$ is a simplex, or the
space of polynomials with degree at most $k$ for each variable if
$K$ is a parallelogram/parallelepiped. Define the finite
element space by
\begin{equation*}
V_{h} :=\{v_{h}\in C(\overline{\Omega}) :v_{h}|_{K}\in P_{k}(K), \ \ \forall K\in\mathcal{T}_{h}\}
\end{equation*}
Furthermore, denote $V_{h}^{0}=V_{h}\cap H_{0}^{1}(\Omega)$.

In the rest of this paper, we denote by $C$ a constant
independent of mesh size with different context in different occurrence, and also use the notation $A\lesssim
F$ to represent $A\leq CF$ with a generic constant $C>0$ independent of mesh size. In addition, $A\approx F$ abbreviates
$A\lesssim F\lesssim A$.

The discrete form reads: Find $(y_{h},z_{h})\in V_{h}\times V_{h}^{0}$ such that
\begin{eqnarray}
    \label{du18}
    (\nabla y_{h},\nabla\psi_{h})=(f,\psi_{h})\ \ \ \ \forall\ \psi_{h}\in V_{h}^{0}, \\
    \label{du19}
    (\nabla z_{h},\nabla\phi_{h})-(\gamma y_{h},\phi_{h})_{\Gamma}-(y_{h},\phi_{h})=-(y_{d},\phi_{h})\ \ \ \ \forall\ \phi_{h}\in V_{h}.
\end{eqnarray}
\begin{theorem}\label{du20}
The discrete variational problem in (\ref{du18})-(\ref{du19}) exists a unique solution $(y_{h},z_{h})\in V_{h}\times V_{h}^{0}$.
\end{theorem}
\begin{proof}
Since the existence of the solution is equivalent to its uniqueness for a finite-dimensional system, it is sufficient to
prove that the corresponding homogeneous system has trivial solution. To this end, let $f=0$ and $y_{d}=0$ in (\ref{du18}) and (\ref{du19}), respectively, we get
\begin{eqnarray}
    \label{du12}
    (\nabla y_{h},\nabla\psi_{h})=0\ \ \ \ \forall\ \psi_{h}\in V_{h}^{0}, \\
    \label{du13}
    (\nabla z_{h},\nabla\phi_{h})-(\gamma y_{h},\phi_{h})_{\Gamma}-(y_{h},\phi_{h})=0\ \ \ \ \forall\ \phi_{h}\in V_{h}.
\end{eqnarray}
Taking $\phi_{h}=y_{h}$ in (\ref{du13}) leads to
\begin{equation}\label{du14}
(\nabla z_{h},\nabla y_{h})-(\gamma y_{h},y_{h})_{\Gamma}-||y_{h}||^{2}=0.
\end{equation}
Noticing $z_{h}\in V_{h}^{0}$ gives $(\nabla z_{h},\nabla y_{h})=0$. Combining this with (\ref{du14}) yields to
\begin{equation*}
\displaystyle\int_{\Gamma}\gamma y_{h}^{2}ds+||y_{h}||^{2}=0,
\end{equation*}
which, altogether with the assumption $\gamma>0$, results in $y_{h}=0$.

(\ref{du13}) with $y_{h}=0$ gives
\begin{equation}\label{du16}
(\nabla z_{h},\nabla \phi_{h})=0\ \ \ \ \forall\ \phi_{h}\in V_{h},
\end{equation}
which, in turn, yields to $(\nabla z_{h},\nabla z_{h})=0$, by choosing $\phi_{h}=z_{h}$. Noticing $z_{h}\in V_{h}^{0}$, we get $z_{h}=0$. Thus, the corresponding homogeneous system has vanishing solution.
\end{proof}

\begin{lemma}\label{Dec13}
Assume that $\theta_{1}^{b}\in V_{h}$ and vanishes at all interior nodes of the mesh, and let $h$ be the size of the quasi-uniform mesh $\mathcal{T}_{h}$. It holds the following estimate
\begin{equation}\label{Dec14}
 ||\nabla\theta_{1}^{b}||\lesssim h^{-1/2}||\theta_{1}^{b}||_{L^{2}(\Gamma)}.
\end{equation}
\end{lemma}
\begin{proof}
Denote $\Omega_{h}^{b}$ the set of element with at least one vertex on the boundary. Since $\theta_{1}^{b}$ vanishes at any node of an
element whose vertices completely contained in the interior of the domain $\Omega$, it's restriction on the element is zero. This implies
that
\begin{equation}\label{Dec14b}
\displaystyle||\nabla\theta_{1}^{b}||^{2}=\sum\limits_{K\in\Omega_{h}^{b}}||\nabla\theta_{1}^{b}||_{K}^{2}.
\end{equation}
For the sake of simplicity, we consider only triangular element in two dimensions as an example, since the similar proof is easily extended
to the other types of element and three-dimensional case. There are three cases as following:

(1). Two vertices of an element $K$ lie on the boundary, i.e., $K$ has an edge $E$ contained in $\Gamma$ ($E=\partial K\cap\Gamma$) (see (Case 1) in Figure \ref{Fig2}). Assuming $||\nabla\theta_{1}^{b}||_{K}=0$ indicates that $\theta_{1}^{b}$ is a constant over the element $K$. And since $\theta_{1}^{b}$ vanishes at internal node of $K$ (there exists at least an internal node such as internal vertex), this shows that $\theta_{1}^{b}$ is zero over $K$, and that $||\nabla\theta_{1}^{b}||_{K} $ is a norm of $\theta_{1}^{b}$ over $K$. Further assuming
$||\theta_{1}^{b}||_{E}=0$, this leads that $\theta_{1}^{b}$ vanishes over $E$, and that $\theta_{1}^{b}$ vanishes at nodes of $E$, and
that $\theta_{1}^{b}$ is zero over $K$. Therefore, $||\theta_{1}^{b}||_{E}$ is another norm of $\theta_{1}^{b}$ over $K$. Since any two
norms are equivalent to each other over a finite-dimension space, we attain
\begin{equation}\label{Dec14c}
||\nabla\theta_{1}^{b}||_{K}\approx C_{K}||\theta_{1}^{b}||_{E},
\end{equation}
where the positive constant $C_{K}$ depends on the size $h_{K}$ of $K$ (and number of dimensions of $V_{h}|_{K}$). To see the dependence
on the size of $K$, we apply scaling argument.

\begin{figure}
    \centering
    \begin{tikzpicture}[scale=1.8]
        \node at (.8,-0.1) {\small{(Case 1)}};
        \draw (.1,.1)--(1.4,.1)--(1.4,1.4)--(.1,1.4)--(.1,.1);
        \draw (.5,.1)--(.7,.5)--(.9,.1);
        \node at (.7,.25) {\scriptsize K};
        \node at (.73,.03) {\scriptsize E};
        \node at (.7,.7) {\scriptsize $\Omega$};
        \node at (.2,.7) {\scriptsize $\Gamma$};

        \node at (2.8,-0.1) {\small{(Case 2)}};
        \draw (2.1,.1)--(3.4,.1)--(3.4,1.4)--(2.1,1.4)--(2.1,.1);
        \draw (2.6,.1)--(2.1,.6)--(2.6,0.6)--(2.6,.1);
        \node at (2.25,.25) {\scriptsize K};
        \node at (2.5,.4) {\scriptsize ${\rm K}^{'}$};
        \node at (2.7,.7) {\scriptsize $\Omega$};
        \node at (2.2,.9) {\scriptsize $\Gamma$};

        \node at (4.8,-0.1) {\small{(Case 3)}};
        \draw (4.1,.1)--(5.4,.1)--(5.4,1.4)--(4.1,1.4)--(4.1,.1);
        \draw (4.7,.1)--(4.3,.5)--(4.3,.1);
        \draw (4.7,.1)--(4.7,.5)--(4.3,.5);
        \draw (4.7,.5)--(5.1,.5)--(4.7,.1);
        \draw (5.1,.5)--(5.1,.1);
        \node at (4.6,.31) {\scriptsize K};
        \node at (4.4,.2) {\scriptsize ${\rm K}^{'}$};
        \node at (4.5,.03) {\scriptsize E};
        \node at (4.7,.7) {\scriptsize $\Omega$};
        \node at (4.2,.7) {\scriptsize $\Gamma$};
        \node at (4.7,.025) {\scriptsize ${\bf x}_{j}$};
    \end{tikzpicture}
    \caption{Three cases of location of an element in $\Omega_{h}^{b}$ for triangular element in two dimensions.}
    \label{Fig2}
\end{figure}
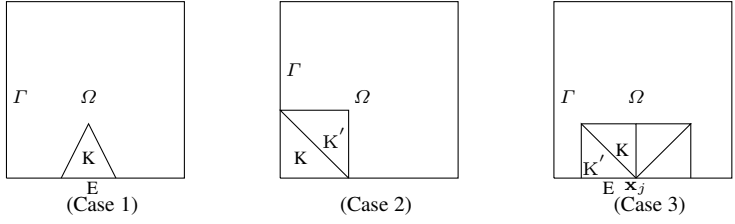

To this end, for any element $K\in\mathcal{T}_{h}$ there exists a bijection $F_{K} :\hat{K}\rightarrow K$, where $\hat{K}$ is the reference element. Denote by $DF_{K}$ the Jacobian matrix and let $J_{K}=|det(DF_{K})|$. It is easy to see that for all element types, the mapping definition
and shape-regularity and quasi-uniformity of the grids imply that
\begin{equation*}
||DF_{K}^{-1}||_{0,\infty, \hat{K}}\approx h_{K}^{-1},\ \ ||J_{K}||_{0,\infty, \hat{K}}\approx h_{K}^{d},\ \ ||DF_{K}||_{0,\infty, \hat{K}}\approx h_{K},
\end{equation*}
which, results in
\begin{equation}\label{Dec16}
\begin{array}{lll}
||\nabla\theta_{1}^{b}||_{K}^{2}&=&\displaystyle\int_{K}\nabla\theta_{1}^{b}\cdot\nabla\theta_{1}^{b}d{\bf x}\vspace{2mm}\\
&=&\displaystyle\int_{\hat{K}}DF_{K}^{-1}\hat{\nabla}\hat{\theta}_{1}^{b}\cdot DF_{K}^{-1}\hat{\nabla}\hat{\theta}_{1}^{b}J_{K}d{\bf\hat{x}}\vspace{2mm}\\
&\approx&\displaystyle h_{K}^{d-2}||\hat{\nabla}\hat{\theta}_{1}^{b}||_{\hat{K}}^{2}.
\end{array}
\end{equation}

Let $E$ be an edge (side) of $K$, and $\hat{E}$ be an edge (side) of $\hat{K}$ with respect to $E$. Similarly, we have
\begin{equation}\label{Dec17}
\begin{array}{lll}
||\theta_{1}^{b}||_{E}^{2}&=&\displaystyle\int_{E}(\theta_{1}^{b})^{2}ds
=\int_{\hat{E}}\frac{|E|}{|\hat{E}|}(\hat{\theta}_{1}^{b})^{2}d\hat{s}\vspace{2mm}\\
&=&\displaystyle\frac{|E|}{|\hat{E}|}||\hat{\theta}_{1}^{b}||_{\hat{E}}^{2}\approx h_{K}^{d-1}||\hat{\theta}_{1}^{b}||_{\hat{E}}^{2}.
\end{array}
\end{equation}
We have from (\ref{Dec14c})
\begin{equation}\label{Dec18}
||\hat{\nabla}\hat{\theta}_{1}^{b}||_{\hat{K}}\approx||\hat{\theta}_{1}^{b}||_{\hat{E}}.
\end{equation}
A combination of (\ref{Dec16}), (\ref{Dec17}), and (\ref{Dec18}) yields to
\begin{equation}\label{Dec19}
||\nabla\theta_{1}^{b}||_{K}\lesssim h_{K}^{-1/2}||\theta_{1}^{b}||_{E}.
\end{equation}

(2). Three vertices of an element $K$ lie on the boundary $\Gamma$, i.e., $K$ has two edges contained in $\Gamma$ (see (Case 2) in Figure \ref{Fig2}). Suppose that one
can always choose an element $K'$ that has an internal vertex and a common edge with $K$. Now consider $\theta_{1}^{b}$ over $K\cup K'$.
Repeating the proof of Case (1), we have
\begin{equation}\label{Dec20}
||\nabla\theta_{1}^{b}||_{K\cup K'}\approx C_{K\cup K'}||\theta_{1}^{b}||_{\Gamma\cap\partial(K\cup K')},
\end{equation}
where $C_{K\cup K'}$ relies on the size $h_{K\cup K'}$, of $K\cup K'$. Using the scaling argument again, we easily obtain
\begin{eqnarray}
\label{Dec21}
||\nabla\theta_{1}^{b}||_{K\cup K'}^{2}\approx h_{K}^{d-2}||\hat{\nabla}\hat{\theta}_{1}^{b}||_{\widehat{K\cup K'}}^{2}\vspace{2mm}\\
\label{Dec22}
||\theta_{1}^{b}||_{\partial(K\cup K')\cap\Gamma}^{2}\approx h_{K}^{d-1}||\hat{\theta}_{1}^{b}||_{\widehat{\partial K\cap\Gamma}}^{2}.
\end{eqnarray}
(\ref{Dec20}) indicates that $||\hat{\nabla}\hat{\theta}_{1}^{b}||_{\widehat{K\cup K'}}\approx||\hat{\theta}_{1}^{b}||_{\widehat{\partial K\cap\Gamma}}$. Hence, we obtain from a combination of (\ref{Dec21}) and (\ref{Dec22})
\begin{equation}\label{Dec23}
||\nabla\theta_{1}^{b}||_{K\cup K'}\lesssim h_{K}^{-1/2}||\theta_{1}^{b}||_{\partial K\cup\Gamma}.
\end{equation}

(3). Only one vertex ${\bf x}_{j}$, of $K$ lies on the boundary $\Gamma$ (see (Case 3) in Figure \ref{Fig2}). Suppose that one can always choose an element $K'$ such that $\partial(K\cup K')$
contains an boundary edge $E$ and $K'$ has a common edge with $K$, i.e., $E\subset\partial(K\cup K')\cap\Gamma$. Similarly to Case (2) or (1), we easily obtain
\begin{equation}\label{Dec24}
||\nabla\theta_{1}^{b}||_{K}\leq||\nabla\theta_{1}^{b}||_{K\cup K'}\lesssim h_{K}^{-1/2}||\theta_{1}^{b}||_{\partial(K\cup K')\cap\Gamma}.
\end{equation}
In fact, in this case, we can also consider $\theta_{1}^{b}$ over the patch $\omega_{{\bf x}_{j}}$ (the set of element
shared ${\bf x}_{j}$ with $K$), of ${\bf x}_{j}$. By using the scaling argument, we can obtain
\begin{equation}\label{Dec25}
||\nabla\theta_{1}^{b}||_{K}\leq||\nabla\theta_{1}^{b}||_{\omega_{{\bf x}_{j}}}\lesssim
h_{K}^{-1/2}||\theta_{1}^{b}||_{\partial(\omega_{{\bf x}_{j}})\cap\Gamma}.
\end{equation}

Collecting (\ref{Dec19}) and (\ref{Dec23})-(\ref{Dec25}), we obtain from (\ref{Dec14b})
\begin{equation*}
\begin{array}{lll}
||\nabla\theta_{1}^{b}||^{2}&=&\displaystyle\sum\limits_{K\in\Omega_{h}^{b}:\ {\rm Case}\ (1)}||\nabla\theta_{1}^{b}||_{K}^{2}+
\sum\limits_{K\in\Omega_{h}^{b}:\ {\rm Case}\ (2)}||\nabla\theta_{1}^{b}||_{K}^{2}+\sum\limits_{K\in\Omega_{h}^{b}:\ {\rm Case}\ (3)}||\nabla\theta_{1}^{b}||_{K}^{2}\vspace{2mm}\\
&\lesssim&\displaystyle\sum\limits_{K\in\Omega_{h}^{b}:\ {\rm Case}\ (1)}h_{K}^{-1}||\theta_{1}^{b}||_{\partial K\cap\Gamma}^{2}+
\sum\limits_{K\in\Omega_{h}^{b}:\ {\rm Case}\ (2)}h_{K}^{-1}||\theta_{1}^{b}||_{\partial K\cap\Gamma}^{2}\vspace{2mm}\\
&\ &\ \displaystyle+\sum\limits_{K\in\Omega_{h}^{b}:\ {\rm Case}\ (3)}h_{K}^{-1}||\theta_{1}^{b}||_{\partial(\omega_{{\bf x}_{j}})\cap\Gamma}^{2}\vspace{2mm}\\
&\lesssim&\displaystyle h^{-1}||\theta_{1}^{b}||_{\Gamma}^{2},
\end{array}
\end{equation*}
which results in the desired estimate (\ref{Dec14}).
\end{proof}

\section{Analysis of error}\label{SECTION5}
Since the control $u$ is equal to the restriction of the state $y$ on the boundary, i.e., $u=y|_{\Gamma}$, it is natural that the restriction
of an approximation $y_{h}$ of $y$ on the boundary is also an approximation of $u$. This shows that $||y-y_{h}||_{0,\Gamma}$ can be used to measure the numerical error of the control, in this sense we write $||u-u_{h}||_{0,\Gamma}=||y-y_{h}||_{0,\Gamma}$.

\begin{theorem}\label{du21}
Assume that $(y,z)\in H^{1}(\Omega)\times H_{0}^{1}(\Omega)$ and $(y_{h},z_{h})\in V_{h}\times V_{h}^{0}$ be the solutions to {\em(\ref{du9})-(\ref{du10})} and {\em(\ref{du18})-(\ref{du19})}, respectively. For $y\in H^{k+1}(\Omega), z\in H^{k+1}(\Omega)\cap H_{0}^{1}(\Omega)$, and for the numerical error of the state variable $y$, there exists a positive constant $C_{\gamma}$ depending on $\gamma$ such that
\begin{equation}\label{du22}
||y-y_{h}||+||\gamma^{1/2}(y-y_{h})||_{0,\Gamma}\leq C_{\gamma}h^{k-1/2}\left(|y|_{k+1}+|z|_{k+1}\right).
\end{equation}
\end{theorem}
\begin{proof}
Denote $R_{h}: H^{1}(\Omega)\rightarrow V_{h}$ the Ritz projection operator by
\begin{equation}\label{Nov1}
(\nabla(R_{h}v),\nabla v_{h})=(\nabla v, \nabla v_{h}),\ \ (v-R_{h}v,1)=0,\ \ \ \forall v_{h}\in V_{h}.
\end{equation}
Recalling the properties of the Ritz projection \cite{Brezz1991,Brenner1994} as following
\begin{equation}\label{Nov2}
||v-R_{h}v||\lesssim h^{k}|v|_{k}, ||\nabla(v-R_{h}v)||\lesssim h^{k-1}|v|_{k},\ \ \forall\ v\in H^{m}(\Omega), 0<k\leq m\leq 3.
\end{equation}
Setting $\eta_{1}=y-R_{h}y$ and $\theta_{1}=R_{h}y-y_{h}$ gives $y-y_{h}=\eta_{1}+\theta_{1}$. We have from triangle inequality
and (\ref{Nov2})
\begin{equation}\label{Nov3}
||y-y_{h}||\leq||\eta_{1}||+||\theta_{1}||\lesssim h^{k+1}|y|_{k+1}+||\theta_{1}||.
\end{equation}

The trace inequality and the properties, (\ref{Nov2}), of the Ritz projection imply that
\begin{equation}\label{du23}
\begin{array}{lll}
||\gamma^{1/2}(y-y_{h})||_{0,\Gamma}&\leq&\gamma^{1/2}||\eta_{1}||_{0,\Gamma}+||\gamma^{1/2}\theta_{1}||_{0,\Gamma}\vspace{2mm}\\
&\lesssim&\gamma^{1/2}||\eta_{1}||^{1/2}||\eta_{1}||_{1}^{1/2}+||\gamma^{1/2}\theta_{1}||_{0,\Gamma}\vspace{2mm}\\
&\lesssim&\gamma^{1/2}\left(||\eta_{1}||+||\nabla\eta_{1}||^{1/2}||\eta_{1}||^{1/2}\right)+||\gamma^{1/2}\theta_{1}||_{0,\Gamma)}\vspace{2mm}\\
&\lesssim&\gamma^{1/2}h^{k+1/2}|y|_{k+1}+||\gamma^{1/2}\theta_{1}||_{0,\Gamma}.
\end{array}
\end{equation}
(\ref{Nov3}) and (\ref{du23}) indicates that we only need to estimate $||\theta_{1}||$ and $||\gamma^{1/2}\theta_{1}||_{0,\Gamma}$ in order to estimate $||y-y_{h}||+||\gamma^{1/2}(y-y_{h})||_{0,\Gamma}$. To this end, let $R_{h}^{0}: H_{0}^{1}(\Omega)\rightarrow V_{h}^{0}$ be the Ritz projection operator by
\begin{equation*}
(\nabla(R_{h}^{0}v),\nabla v_{h})=(\nabla v,\nabla v_{h})\ \ \ \ \forall\ v_{h}\in V_{h}^{0}.
\end{equation*}
Again recalling the properties of the Ritz projection \cite{Brezz1991,Brenner1994} as following
\begin{equation}\label{du34}
||\nabla(v-R_{h}^{0}v)||\lesssim h^{k-1}|v|_{k},\ \ \ \forall\ v\in H^{m}(\Omega)\cap H_{0}^{1}(\Omega), 0<k\leq m\leq 3.
\end{equation}

Setting $\eta_{2}=z-R_{h}^{0}z,\theta_{2}=R_{h}^{0}z-z_{h}$ gives $z-z_{h}=\eta_{2}+\theta_{2}$. From (\ref{du10}) and (\ref{du19}), we obtain the following orthogonality
\begin{equation}\label{du25}
(\nabla(z-z_{h}),\nabla\phi_{h})-(\gamma(y-y_{h}),\phi_{h})_{\Gamma}-(y-y_{h},\phi_{h})=0,\ \ \forall\ \phi_{h}\in V_{h}.
\end{equation}
Especially taking $\phi_{h}=\theta_{1}\in V_{h}$ in (\ref{du25}) yields to
\begin{equation*}
(\nabla\eta_{2}+\nabla\theta_{2},\nabla\theta_{1})-(\gamma(\eta_{1}+\theta_{1}),\theta_{1})_{\Gamma}-(\eta_{1}+\theta_{1},\theta_{1})=0,
\end{equation*}
which results in
\begin{equation}\label{du26}
||\gamma^{1/2}\theta_{1}||_{0,\Gamma}^{2}+||\theta_{1}||^{2}=(\nabla\eta_{2},\nabla\theta_{1})+
(\nabla\theta_{2},\nabla\theta_{1})-(\gamma\eta_{1},\theta_{1})_{\Gamma}-(\eta_{1},\theta_{1})
\end{equation}

From (\ref{du9}) and (\ref{du18}), we get the following orthogonal property
\begin{equation}\label{du27}
(\nabla(y-y_{h}),\nabla\psi_{h})=0,\ \ \ \forall\ \psi_{h}\in V_{h}^{0}.
\end{equation}
Taking $\psi_{h}=\theta_{2}\in V_{h}^{0}$ in (\ref{du27}) yields to
\begin{equation}\label{du28}
(\nabla\theta_{2},\nabla\theta_{1})=-(\nabla\eta_{1},\nabla\theta_{2})=0.
\end{equation}
In the second step above, we apply the orthogonal property of the Ritz projection, because of $\theta_{2}\in V_{h}^{0}\subset V_{h}$.
Combining (\ref{du26}) with (\ref{du28}), we attain
\begin{equation}\label{du29}
||\gamma^{1/2}\theta_{1}||_{0,\Gamma}^{2}+||\theta_{1}||^{2}=(\nabla\eta_{2},\nabla\theta_{1})
-(\gamma\eta_{1},\theta_{1})_{\Gamma}-(\eta_{1},\theta_{1}).
\end{equation}

In what follows, we estimate each term on the right-hand side of (\ref{du29}). In terms of the proof of (\ref{Nov3}) and (\ref{du23}),
we immediately obtain the estimates of the last two terms on the right-hand side of (\ref{du29})
\begin{equation}\label{du30}
|-(\eta_{1},\theta_{1})|\lesssim h^{k+1}|y|_{k+1}||\theta_{1}||
\end{equation}
and
\begin{equation}\label{du31}
|-(\gamma\eta_{1},\theta_{1})_{\Gamma}| \lesssim \gamma^{1/2}h^{k+1/2}|y|_{k+1}||\gamma^{1/2}\theta_{1}||_{0,\Gamma}.
\end{equation}

To estimate the first term of on the right-hand side of (\ref{du29}), we decompose $\theta_{1}$ into $\theta_{1}^{i}$ and $\theta_{1}^{b}$,
where the value of $\theta_{1}^{i}$ at the internal node equals to the one of $\theta_{1}$ at the corresponding node, and
the value of $\theta_{1}^{i}$ at the boundary node is zero; the value of $\theta_{1}^{b}$ at the internal node is zero, and the value of
$\theta_{1}^{b}$ at boundary node equals to the one of
$\theta_{1}$ at the corresponding node. Obviously, $\theta_{1}=\theta_{1}^{i}+\theta_{1}^{b}$.

Noticing $\theta_{1}^{i}\in V_{h}^{0},\theta_{1}^{b}\in V_{h}$, we have from the definition of the Ritz projection
\begin{equation}\label{du32}
\begin{array}{lll}
(\nabla\eta_{2},\nabla\theta_{1})&=&(\nabla\eta_{2},\nabla\theta_{1}^{i}+\nabla\theta_{1}^{b})\vspace{2mm}\\
&=&(\nabla\eta_{2},\nabla\theta_{1}^{b})\leq\|\nabla\eta_{2}\|\,\|\nabla\theta_{1}^{b}\|.
\end{array}
\end{equation}
We further derive from Lemma \ref{Dec13}, together with $\theta_{1}^{b}=\theta_{1}$ on the boundary $\Gamma$
\begin{equation}\label{du33}
\begin{array}{lll}
(\nabla\eta_{2},\nabla\theta_{1})&\lesssim& h^{-1/2}\|\nabla\eta_{2}\|\,\|\theta_{1}^{b}\|_{0,\Gamma}\vspace{2mm}\\
&=&h^{-1/2}\gamma^{-1/2}\|\nabla\eta_{2}\|\,\|\gamma^{1/2}\theta_{1}\|_{0,\Gamma}.
\end{array}
\end{equation}

By combining (\ref{du29})-(\ref{du31}) with (\ref{du33}), and applying the properties, (\ref{du34}), of the Ritz projection, and
Young inequality, we obtain
\begin{equation*}
\begin{array}{lll}
||\gamma^{1/2}\theta_{1}||_{0,\Gamma}^{2}+||\theta_{1}||^{2}&\leq& Ch^{-1/2}\gamma^{-1/2}\|\nabla\eta_{2}\|\,\|\gamma^{1/2}\theta_{1}\|_{0,\Gamma}\vspace{2mm}\\
&\ &\ +Ch^{k+1}|y|_{k+1}||\theta_{1}||+C\gamma^{1/2}h^{k+1/2}|y|_{k+1}||\gamma^{1/2}\theta_{1}||_{0,\Gamma}\vspace{2mm}\\
&\leq&Ch^{-1/2}\gamma^{-1/2}h^{k}|z|_{k+1}||\gamma^{1/2}\theta_{1}||_{0,\Gamma}\vspace{2mm}\\
&\ &\ +Ch^{k+1}|y|_{k+1}||\theta_{1}||+C\gamma^{1/2}h^{k+1/2}|y|_{k+1}||\gamma^{1/2}\theta_{1}||_{0,\Gamma}\vspace{2mm}\\
&\leq&C\gamma^{-1}h^{2k-1}|z|_{k+1}^{2}+||\gamma^{1/2}\theta_{1}||_{0,\Gamma}^{2}/4+Ch^{2(k+1)}|y|_{k+1}^{2}\vspace{2mm}\\
&\ &\ +||\theta_{1}||^{2}/2+C\gamma h^{2k+1}|y|_{k+1}^{2}+||\gamma^{1/2}\theta_{1}||_{0,\Gamma}^{2}/4,
\end{array}
\end{equation*}
which, implies
\begin{equation}\label{du36}
||\gamma^{1/2}\theta_{1}||_{0,\Gamma}^{2}+||\theta_{1}||^{2}\leq
C_{\gamma}h^{2k-1}\left(|z|_{k+1}^{2}+|y|_{k+1}^{2}\right).
\end{equation}

Collecting (\ref{Nov3}), (\ref{du23}), and (\ref{du36}), we get
\begin{equation*}
||\gamma^{1/2}(y-y_{h})||_{0,\Gamma}^{2}+||y-y_{h}||^{2}\leq
C_{\gamma}h^{2k-1}\left(|z|_{k+1}^{2}+|y|_{k+1}^{2}\right),
\end{equation*}
which results in the desired estimate (\ref{du22}).
\end{proof}

\begin{theorem}\label{du53a}
Assume that $(y,z)\in H^{1}(\Omega)\times H_{0}^{1}(\Omega)$ and $(y_{h},z_{h})\in V_{h}\times V_{h}^{0}$ be the solutions to {\em(\ref{du9})-(\ref{du10})} and {\em(\ref{du18})-(\ref{du19})}, respectively. For $y\in H^{k}(\Omega), z\in H^{k+1}(\Omega)\cap H_{0}^{1}(\Omega)$, and for the numerical error of the state variable $y$, there exists a positive constant $C_{\gamma}$ depending on $\gamma$ such that
\begin{equation}\label{du53b}
||y-y_{h}||+||\gamma^{1/2}(y-y_{h})||_{0,\Gamma}\leq C_{\gamma}h^{k-1/2}\left(|y|_{k}+|z|_{k+1}\right).
\end{equation}
\end{theorem}
\begin{proof}
Following the idea of the proof in Theorem \ref{du21}, using $y\in H^{k}(\Omega)$ instead of $y\in H^{k+1}(\Omega)$ while concerning the Ritz projection of $y$, we obtain the desired estimate (\ref{du53b}).
\end{proof}

\begin{remark}\label{du37}
As pointed at the beginning of this section, we understand $||u-u_{h}||_{L^{2}(\Gamma)}$ as $||y-y_{h}||_{0,\Gamma}$.
For $y\in H^{k+1}(\Omega),z\in H_{0}^{1}(\Omega)\cap H^{k+1}(\Omega)$, Theorem \ref{du21} gives the control an estimate
\begin{equation*}
||u-u_{h}||_{0,\Gamma}\leq C_{\gamma}h^{k-1/2}\left(|y|_{k+1}+|z|_{k+1}\right);
\end{equation*}
For $y\in H^{k}(\Omega),z\in H_{0}^{1}(\Omega)\cap H^{k+1}(\Omega)$, Theorem \ref{du53a} gives the control an estimate
\begin{equation*}
||u-u_{h}||_{0,\Gamma}\leq C_{\gamma}h^{k-1/2}\left(|y|_{k}+|z|_{k+1}\right).
\end{equation*}
\end{remark}

\begin{theorem}\label{du38}
Assume that $(y,z)\in H^{1}(\Omega)\times H_{0}^{1}(\Omega)$ and $(y_{h},z_{h})\in V_{h}\times V_{h}^{0}$ be the solutions to {\em(\ref{du9})-(\ref{du10})} and {\em(\ref{du18})-(\ref{du19})}, respectively. For $y\in H^{k+1}(\Omega), z\in H^{k+1}(\Omega)\cap H_{0}^{1}(\Omega)$, the numerical errors of the adjoint state $z$ is bounded by
\begin{equation}\label{du39}
||\nabla(z-z_{h})||\leq C_{\gamma}h^{k-1/2}(|y|_{k+1}+|z|_{k+1});
\end{equation}
For $y\in H^{k}(\Omega), z\in H^{k+1}(\Omega)\cap H_{0}^{1}(\Omega)$, the numerical errors of the adjoint state $z$ is bounded by
\begin{equation}\label{du53c}
||\nabla(z-z_{h})||\leq C_{\gamma}h^{k-1/2}\left(|y|_{k}+|z|_{k+1}\right).
\end{equation}
\end{theorem}
\begin{proof}
 Recalling the decomposition of the error $z-z_{h}$ in the proof of Theorem \ref{du21}, we obtain from the orthogonal property of the Ritz projection
\begin{equation}\label{du41}
||\nabla(z-z_{h})||^{2}=||\nabla\eta_{2}||^{2}+||\nabla\theta_{2}||^{2},
\end{equation}
which, together with the property (\ref{du34}) of the Ritz projection, results in,
\begin{equation}\label{du24}
||\nabla(z-z_{h})||\lesssim h^{k}|z|_{k+1}+||\nabla\theta_{2}||.
\end{equation}
The inequality (\ref{du24}) means that it is sufficient to only estimate $||\nabla\theta_{2}||$ in order to estimate $||\nabla(z-z_{h})||$.

Taking $\phi_{h}=\theta_{2}\in V_{h}^{0}$ in (\ref{du25}) yields to
\begin{equation*}
(\nabla\eta_{2}+\nabla\theta_{2},\nabla\theta_{2})-(y-y_{h},\theta_{2})=0,
\end{equation*}
which, together with the orthogonal relation $(\nabla\eta_{2},\nabla\theta_{2})=0$, results in,
\begin{equation}\label{du40}
||\nabla\theta_{2}||^{2}=(y-y_{h},\theta_{2})\leq||y-y_{h}||||\theta_{2}||.
\end{equation}
Applying the Poincar\'{e} inequality, we obtain from (\ref{du40})
\begin{equation}\label{du42}
||\nabla\theta_{2}||\lesssim||y-y_{h}||.
\end{equation}

Combing (\ref{du24}) with (\ref{du42}), we obtain
\begin{equation}\label{du43}
||\nabla(z-z_{h})||\lesssim h^{k}|z|_{k+1}+||y-y_{h}||,
\end{equation}
which, together with (\ref{du22}) and (\ref{du53b}), respectively, results in the desired estimates (\ref{du39}) and (\ref{du53c}).
\end{proof}

\begin{remark}
Lemma \ref{Febsaa} suggests that the $H^{2}$ regularity for the state cannot be reached on polygonal$/$polyhedral domain. This makes these
estimates restricted to the case of $k=1$. However, the $H^{k}$ regularity for the state can be reached for domains with sufficiently smooth
boundary. Since the Dirichlet boundary control problem is completely different from the Dirichlet boundary value problem, it is non-trivial to
generalize analytical technique for high order element (including isoparametric-equivalent element) for the Dirichlet boundary value problem
to the Dirichlet boundary control problem. Here are two remedies in two dimensional case for the sake of simplicity.

In first case, let $\Omega\subset\mathbb{R}^{2}$ be a bounded domain with smooth boundary and $\mathcal{T}_{h}$ be a ``triangulation"
of $\Omega$, where each triangle at the boundary has at most one curved side. The finite element spaces $V_{h}$ and $V_{h}^{0}$ are defined by
\begin{equation*}
V_{h}=\Big\{v\in C(\bar{\Omega}): v|_{K}\in P_{k}(K), \forall {K\in\mathcal{T}_{h}}\Big\}\ \ {\rm and}\ \ V_{h}^{0}=V_{h}\cap H_{0}^{1}(\Omega){\rm, respectively.}
\end{equation*}
By using standard interpolation error estimates, we can easily verify that the properties of the Ritz projection on $V_{h}$ (and $V_{h}^{0}$) are still true. Assume that the ``triangulation" $\mathcal{T}_{h}$ guarantees Lemma \ref{Dec13}. Indeed, this is easily realised by assuming
that there exists $\rho>0$ such that for each triangle $T\in\mathcal{T}_{h}$ one can find two concentric circular discs $D_{1}$ and $D_{2}$ such that
\begin{equation*}
D_{1}\subseteq T\subseteq D_{2}\ \ \ {\rm and}\ \ \ \displaystyle\frac{{\rm diam}D_{2}}{{\rm diam}D_{1}}\leq\rho.
\end{equation*}
Since $\partial\Omega$ is smooth, for h small enough, we have $h_{e}<2{\rm diam}T<2{\rm diam}D_{2}$ (curved side $e\subset\partial T$, $h_{e}$ denote the arc length of $e$). This indicates Lemma \ref{Dec13} is still valid. Therefore, the results of Theorems \ref{du21}-\ref{du38} are applicable to high order curved-triangle Lagrange element.  

In the second case, recall that we have a polyhedral approximation, $\Omega_{h}$ to $\Omega$, and an isoparametric mapping $F^{h}$ such that
$F^{h}(\Omega_{h})$ closely approximates to $\Omega$, and denote $\tilde{V}_{h}$ a base finite element space defined on $\Omega_{h}$, the resulting space,
\begin{equation*}
V_{h} :=\big\{v((F^{h})^{-1}({\bf x})) : {\bf x}\in F^{h}(\Omega_{h}),\ v\in \tilde{V}_{h}\big\},
\end{equation*}
is an isoparametric-equivalent finite element space (we refer to \cite{Ciarlet} on details). Let $V_{h}^{0}=V_{h}\cap H_{0}^{1}(F^{h}(\Omega_{h}))$. If we impose the control rule on
$\partial(F^{h}(\Omega_{h}))$, i.e., $\displaystyle\frac{1}{\gamma}\frac{\partial z}{\partial{\bf n}}=u$ on $\partial(F^{h}(\Omega_{h}))$, this
shows that we are considering the problem (\ref{du1})-(\ref{du3}) on the domain $F^{h}(\Omega_{h})$. The only difference is that we substitute 
the domain $\Omega$ in the precious context with $F^{h}(\Omega_{h})$. Since the corresponding Ritz 
projection $R_{h}: H^{1}(F^{h}(\Omega_{h}))\rightarrow V_{h}$ ($R_{h}^{0}: H_{0}^{1}(F^{h}(\Omega_{h}))\rightarrow V_{h}^{0}$) still possesses the same approximation properties as (\ref{Nov2})((\ref{du34})), and since the result of Lemma \ref{Dec13} can be achieved by the similar proof. Therefore, by repeating the proof of Theorem \ref{du21}, we can obtain the following estimate
\begin{equation*}
||y-y_{h}||_{0,F^{h}(\Omega_{h})}+||\gamma^{1/2}(y-y_{h})||_{0,\partial(F^{h}(\Omega_{h}))}\leq 
C_{\gamma}h^{k-\frac{1}{2}}(|y|_{k,F^{h}(\Omega_{h})}+|z|_{k+1,F^{h}(\Omega_{h})})
\end{equation*}
under the assumption that $y\in H^{k}(F^{h}(\Omega_{h})), z\in H^{k+1}(F^{h}(\Omega_{h}))$.

Furthermore, we will assume there is auxiliary mapping $F: \Omega_{h}\rightarrow\Omega$ and that $F_{i}^{h}=I^{h}F_{i}$ for each component of 
the mapping. Here $I^{h}v$ denotes the isoparametric interpolation by $I^{h}v(F^{h}({\bf x}))=\tilde{I}^{h}\tilde{v}({\bf x})$ for all ${\bf x}\in\Omega_{h}$ where $\tilde{v}({\bf x})=v(F^{h}({\bf x}))$ for all ${\bf x}\in\Omega_{h}$ and $\tilde{I}_{h}$ is the global interpolation for the base finite element space, $\tilde{V}_{h}$ (we refer to \cite{Ciarlet} on details). Thus, the 
mapping $\Phi^{h}: \Omega\rightarrow F^{h}(\Omega_{h})$ defined by $\Phi^{h}({\bf x})=F^{h}(F^{-1}({\bf x}))$, suggests that $y$ regarded as a function in $F^{h}(\Omega_{h})$ possesses the same regularity as $J_{\Phi^{h}}^{-1}$ (inverse matrix of the Jacobian $J_{\Phi^{h}}$)  when $y$ is smooth enough in the domain $\Omega$,  this can easily be observed by the chain rule. Therefore, the key is the regularity of the inverse mapping $F^{-1}$ , because the regularity of $F^{h}$ may be reached by using isoparametric interpolation operator of high order.  Unfortunately, in mapping a polyhedral domain to a smooth domain, a $C^{1}$ mapping is inappropriate. However, since $F^{h}(\Omega_{h})$ closely approximates to $\Omega$, and $\partial(F^{h}(\Omega_{h}))$ consists of curved sides,  it is certain that $y$ regarded as a function in $F^{h}(\Omega_{h})$ has higher regularity than $y$ regarded as a function in $\Omega_{h}$ . This shows that the regularity of $y$ in the domain $F^{h}(\Omega_{h})$ can be reached asymptotically. Of course, the construction of  such a mapping $F$ is non-trivial , but it is done in \cite{Lenoir}.
\end{remark}

It is well known that the $L^{2}$ norm of numerical error is controlled by the $H^{1}$ norm for conforming finite element approximation to the standard Laplacian equation, and that the $L^{2}$ norm of numerical error is of order one higher than the $H^{1}$ norm. The following Theorem \ref{du45} shows that $||y-y_{h}||$ is still controlled by $||y-y_{h}||_{1}$, but isn't of order one higher than $||y-y_{h}||_{1}$. This will be testified by numerical experiments in Section\ref{SECTION6}.

\begin{theorem}\label{du45}
Assume that $(y,z)\in H^{1}(\Omega)\times H_{0}^{1}(\Omega)$ and $(y_{h},z_{h})\in V_{h}\times V_{h}^{0}$ be the solutions to {\em(\ref{du9})-(\ref{du10})} and {\em(\ref{du18})-(\ref{du19})}, respectively. It holds
\begin{equation}\label{du46}
||y-y_{h}||\lesssim||\nabla(y-y_{h}||.
\end{equation}
\end{theorem}
\begin{proof}
Consider the following Neumann boundary-value problem
\begin{equation}\label{Nov11}
    \left\{
    \begin{array}{ll}
        -\triangle w=y({\bf x})-y_{h}({\bf x}) & \mbox{in}~\Omega,\\
        \displaystyle\frac{\partial w}{\partial{\bf n}}=\gamma(y({\bf x})-y_{h}({\bf x}))|_{\Gamma} & \mbox{on}~\Gamma.
     \end{array}
     \right.
\end{equation}
The continuous weak formulation for the problem (\ref{Nov11}) reads: Find $w\in H^{1}(\Omega)$ such that
\begin{equation}\label{Nov11b}
(\nabla w,\nabla\psi)=(\gamma(y-y_{h}),\psi)_{\Gamma}+(y-y_{h},\psi)\ \ \ \forall\ \psi\in H^{1}(\Omega).
\end{equation}

We get the following orthogonality from a combination of (\ref{du10}) and (\ref{du19})
\begin{equation*}
(\nabla(z-z_{h}),\nabla v_{h})-(\gamma(y-y_{h}),v_{h})_{\Gamma}-(y-y_{h},v_{h})=0,\ \ \forall\ v_{h}\in V_{h}.
\end{equation*}
Owing to $v_{h}=1\in V_{h}$, the above identity implies that
\begin{equation*}
\displaystyle\int_{\Omega}(y({\bf x})-y_{h}({\bf x}))d{\bf x}+\int_{\Gamma}\gamma(y({\bf x})-y_{h}({\bf x}))ds=0.
\end{equation*}
This shows that the problem (\ref{Nov11}) satisfies the consistent condition. Therefore, the weak formulation (\ref{Nov11b}) has a unique solution in the sense that the solutions differ by a constant, and satisfies the following estimate
\begin{equation}\label{Nov12}
||\nabla w|| \lesssim ||y-y_{h}||+\gamma||y-y_{h}||_{-1/2,\Gamma}.
\end{equation}

Taking $\psi=y$ and $\psi=y_{h}$, respectively, in (\ref{Nov11b}) yields to
\begin{equation}\label{Nov11c}
(\nabla w,\nabla y)=(\gamma(y-y_{h}),y)_{\Gamma}+(y-y_{h},y)
\end{equation}
and
\begin{equation}\label{Nov11d}
(\nabla w,\nabla y_{h})=(\gamma(y-y_{h}),y_{h})_{\Gamma}+(y-y_{h},y_{h}).
\end{equation}
A combination of (\ref{Nov11c}) and (\ref{Nov11d}) leads to
\begin{equation}\label{Nov11e}
||\gamma^{1/2}(y-y_{h})||_{0,\Gamma}^{2}+||y-y_{h}||^{2}=(\nabla w,\nabla(y-y_{h}))\leq\|\nabla w\|\,\|\nabla(y-y_{h})\|.
\end{equation}

We obtain from (\ref{Nov12})
\begin{equation}\label{Nov13}
||\nabla w||\lesssim||y-y_{h}||+\gamma||y-y_{h}||_{0,\Gamma}.
\end{equation}
A combination (\ref{Nov11e}) and (\ref{Nov13}) yields to
\begin{equation*}
||\gamma^{1/2}(y-y_{h})||_{0,\Gamma}^{2}+||y-y_{h}||^{2}\lesssim(||y-y_{h}||+\gamma||y-y_{h}||_{0,\Gamma})||\nabla(y-y_{h})||,
\end{equation*}
which, results in
\begin{equation*}
||y-y_{h}||\leq\big(||\gamma^{1/2}(y-y_{h})||_{0,\Gamma}^{2}+||y-y_{h}||^{2}\big)^{1/2}\lesssim ||\nabla(y-y_{h})||.
\end{equation*}
we complete the proof of (\ref{du46}).
\end{proof}

\begin{remark}
In terms of the proof of Theorem \ref{du45}, for a function $v\in H^{1}(\Omega)$ satisfying
\begin{equation*}
\displaystyle\int_{\Omega} v d{\bf x}+\int_{\Gamma} v ds=0,
\end{equation*}
it holds an analogue of the Poincar\'{e} inequality
\begin{equation*}
||v||_{1}\lesssim||\nabla v||.
\end{equation*}
\end{remark}

\section{Stability for discrete solution}\label{SECTION6A}
Since the control is firstly concerned in practice for the optimal control problem, this section specially devotes to an analysis of the stability for the control in the sense that the restriction of the discrete state $y_{h}$ on the boundary is an approximation of the control $u$. To this end, let $V_{h}^{\partial}$ be the trace space corresponding to $V_{h}$, i.e., $V_{h}^{\partial}=V_{h}|_{\Gamma}$. Recall the following ``inverse estimate" for finite element
functions $\chi_{h}\in V_{h}^{\partial}$:
\begin{equation}\label{Feb1}
|\chi_{h}|_{H^{1/2}(\Gamma)}\lesssim h^{-1/2}||\chi_{h}||_{L^{2}(\Gamma)}.
\end{equation}
Indeed, this can be found in \cite{May2013} or be proven by combining estimates in \cite{Ciarlet,Brenner1994} with standard results from interpolation theory. We define the $L^{2}$ projection $P_{h}^{\partial}: L^{2}(\Gamma)\rightarrow V_{h}^{\partial}$ by
\begin{equation*}
(q-P_{h}^{\partial}q,\chi_{h})=0,\ \ \forall\chi_{h}\in V_{h}^{\partial}.
\end{equation*}
By standard results for finite element elements we have the error estimate (see \cite{Ciarlet,Brenner1994,Casa2006})
\begin{equation}\label{Feb2}
||q-P_{h}^{\partial}q||_{0,\Gamma}+h^{1/2}|P_{h}^{\partial}q|_{1/2,\Gamma}\lesssim h^{1/2}|q|_{1/2,\Gamma},\ \ \forall q\in H^{1/2}(\Gamma).
\end{equation}

\begin{theorem}\label{Feb2a}
Assume that $f\in H^{-1}(\Omega),y_{d}\in L^{2}(\Omega)$, the domain $\Omega$ is convex, and its boundary $\Gamma$ is Lipschitz continuous.
There exists a positive constant $C_{\gamma}$ depending on $\gamma$ such that
\begin{equation}\label{Feb3}
||\gamma^{1/2}y_{h}||_{0,\Gamma}+||y_{h}||\leq C_{\gamma}\left(||f||_{-1}+||y_{d}||\right).
\end{equation}
\end{theorem}
\begin{proof}
Taking $\phi_{h}=y_{h}$ and $\psi_{h}=z_{h}$ in (\ref{du19}) and (\ref{du18}), respectively, gives
\begin{equation}\label{Feb4}
\begin{array}{lll}
||\gamma^{1/2}y_{h}||_{0,\Gamma}^{2}+||y_{h}||^{2}&=&(\nabla z_{h},\nabla y_{h})+(y_{d},y_{h})\vspace{2mm}\\
&=&(f,z_{h})+(y_{d},y_{h})\vspace{2mm}\\
&=&(f,z_{h}-z)+(f,z)+(y_{d},y_{h})\vspace{2mm}\\
&\leq&||f|_{-1}|||z-z_{h}||_{1}+||f||_{-1}||z||_{1}+\|y_{d}\|\,\|y_{h}\|.
\end{array}
\end{equation}
Noticing $z-z_{h}\in H_{0}^{1}(\Omega)$, we obtain from the Poincar\'{e} inequality, (\ref{du53c}) with $k=1$, and (\ref{Dec5})
\begin{equation}\label{Feb5}
\begin{array}{lll}
||z-z_{h}||_{1}&\lesssim&||\nabla(z-z_{h})||\vspace{2mm}\\
&\leq& C_{\gamma}h^{1/2}\left(||y||_{1}+||z||_{2}\right)\vspace{2mm}\\
&\leq& C_{\gamma}h^{1/2}\left(||y||_{1}+||y-y_{d}||\right)\vspace{2mm}\\
&\leq& C_{\gamma}h^{1/2}\left(||y||_{1}+||y_{d}||\right).
\end{array}
\end{equation}

Combining (\ref{Feb4}) with (\ref{Feb5}), together with Young inequality , gives
\begin{equation}\label{Feb6}
||\gamma^{1/2}y_{h}||_{0,\Gamma}^{2}+||y_{h}||^{2}\leq C_{\gamma}\left(||y||_{1}^{2}+||y_{d}||^{2}+||f||_{-1}^{2}+||z||_{1}^{2}\right).
\end{equation}
Applying the stable estimates (\ref{dudu89}) and (\ref{Dec1}) of the state $y$ and adjoint state $z$, respectively, we get
\begin{equation*}
||\gamma^{1/2}y_{h}||_{0,\Gamma}^{2}+||y_{h}||^{2}\leq C_{\gamma}\left(||y_{d}||^{2}+||f||_{-1}^{2}\right),
\end{equation*}
which, results in the desired estimate (\ref{Feb3}).
\end{proof}

\begin{theorem}
Under the assumption of Theorem \ref{Feb2a}, the discrete solutions admit the uniform bound
\begin{equation}\label{Feb7}
|y_{h}|_{1/2,\Gamma}\leq C_{\gamma}\left(||f||_{-1}+||y_{d}||\right).
\end{equation}
\end{theorem}
\begin{proof}
From triangle inequality, ``inverse estimate" (\ref{Feb1}), and the property, (\ref{Feb2}), of the $L^{2}$ projection operator $P_{h}^{\partial}$, we get
\begin{equation}\label{Feb8}
\begin{array}{lll}
|y_{h}|_{1/2,\Gamma}&\leq&|y_{h}-P_{h}^{\partial}y|_{1/2,\Gamma}+|P_{h}^{\partial}y-y|_{1/2,\Gamma}+
|y|_{1/2,\Gamma}\vspace{2mm}\\
&\lesssim&h^{-1/2}||y_{h}-P_{h}^{\partial}y||_{0,\Gamma}+|y|_{1/2,\Gamma}\vspace{2mm}\\
&\leq&h^{-1/2}\left(||y_{h}-y||_{0,\Gamma}+||y-P_{h}^{\partial}y||_{0,\Gamma}\right)+|y|_{1/2,\Gamma}\vspace{2mm}\\
&\lesssim&h^{-1/2}||y_{h}-y||_{0,\Gamma}+h^{-1/2}h^{1/2}|y|_{1/2,\Gamma}+|y|_{1/2,\Gamma)}.
\end{array}
\end{equation}

From (\ref{du53b}) with $k=1$ and (\ref{Dec5}), we have
\begin{equation}\label{Feb9}
||y_{h}-y||_{0,\Gamma}\leq C_{\gamma}h^{1/2}\left(||\nabla y||+|z|_{2}\right)\leq C_{\gamma}h^{1/2}\left(||y||_{1}+||y_{d}||\right).
\end{equation}
A combination (\ref{Feb8}) and (\ref{Feb9}) yields to
\begin{equation}\label{Feb10}
|y_{h}|_{1/2,\Gamma}\leq C_{\gamma}\left(||y||_{1}+||y_{d}||+|y|_{1/2,\Gamma}\right)\leq C_{\gamma}\left(||y||_{1}+||y_{d}||\right).
\end{equation}
The desired estimate (\ref{Feb7}) follows from a combination of (\ref{Feb10}), (\ref{dudu89}) and (\ref{Dec1}).
\end{proof}

\section{Numerical experiments}\label{SECTION6}
In this section, we test the performance of finite element approximation to the variational formulation developed in this paper with two model problems. The actual solution of the first model problem is known, and the true solution of the second example is unknown, and two settings of the regularization parameter $\gamma$ will be considered here, We are thus able to study the convergence rate of the state $y$ and adjoint state $z$, as well as the control variable $u$ over quasi-uniform mesh, and to study the relation between the singularity of the actual solution and the regularization parameter in Example two. Note that we shall employ piecewise linear element in both examples. Let $\{\psi_{i}\}$ and $\{\phi_{j}\}$ be respectively the basis of $V_{h}^{0}$ and $V_{h}$, then the algebraic system with respect to (\ref{du18})-(\ref{du19})
has the following form
\begin{equation*}
    \left(\begin{array}{cc} A & O \\ B & C \end{array}\right)
    \left(\begin{array}{c} Y \\ Z \end{array}\right) =
    \left(\begin{array}{c} F \\ G \end{array}\right).
\end{equation*}

\subsection{Example one}
We consider the problem (\ref{du1})-(\ref{du2}) over a unit square $\Omega=(0,1)\times(0,1)$ with
\begin{equation*}
\displaystyle f=-\frac{4}{\gamma}, y_{d}=\left(2+\frac{1}{\gamma}\right)\left(x_{1}^2-x_{1}+x_{2}^2-x_{2}\right).
\end{equation*}
The exact solutions are given by
\begin{equation*}
\displaystyle u=\frac{x_{1}^2-x_{1}+x_{2}^2-x_{2}}{\gamma}, y=\frac{x_{1}^2-x_{1}+x_{2}^2-x_{2}}{\gamma}, z=\left(x_{1}^2-x_{1}\right)\left(x_{2}^2-x_{2}\right).
\end{equation*}
It is easy to verify that the control $u$, state $y$, and adjoint state $z$ satisfy
\begin{equation*}
\displaystyle u=y|_{\Gamma}=\frac{1}{\gamma}\frac{\partial z}{\partial{\bf n}}\Big|_{\Gamma}.
\end{equation*}
Here, we consider two settings, $\gamma=1$ and $\gamma=0.01$, of regularization parameter.

\begin{figure}[t]
    \begin{minipage}[t]{0.5\linewidth}
        \centering
        \includegraphics[width=2.25in]{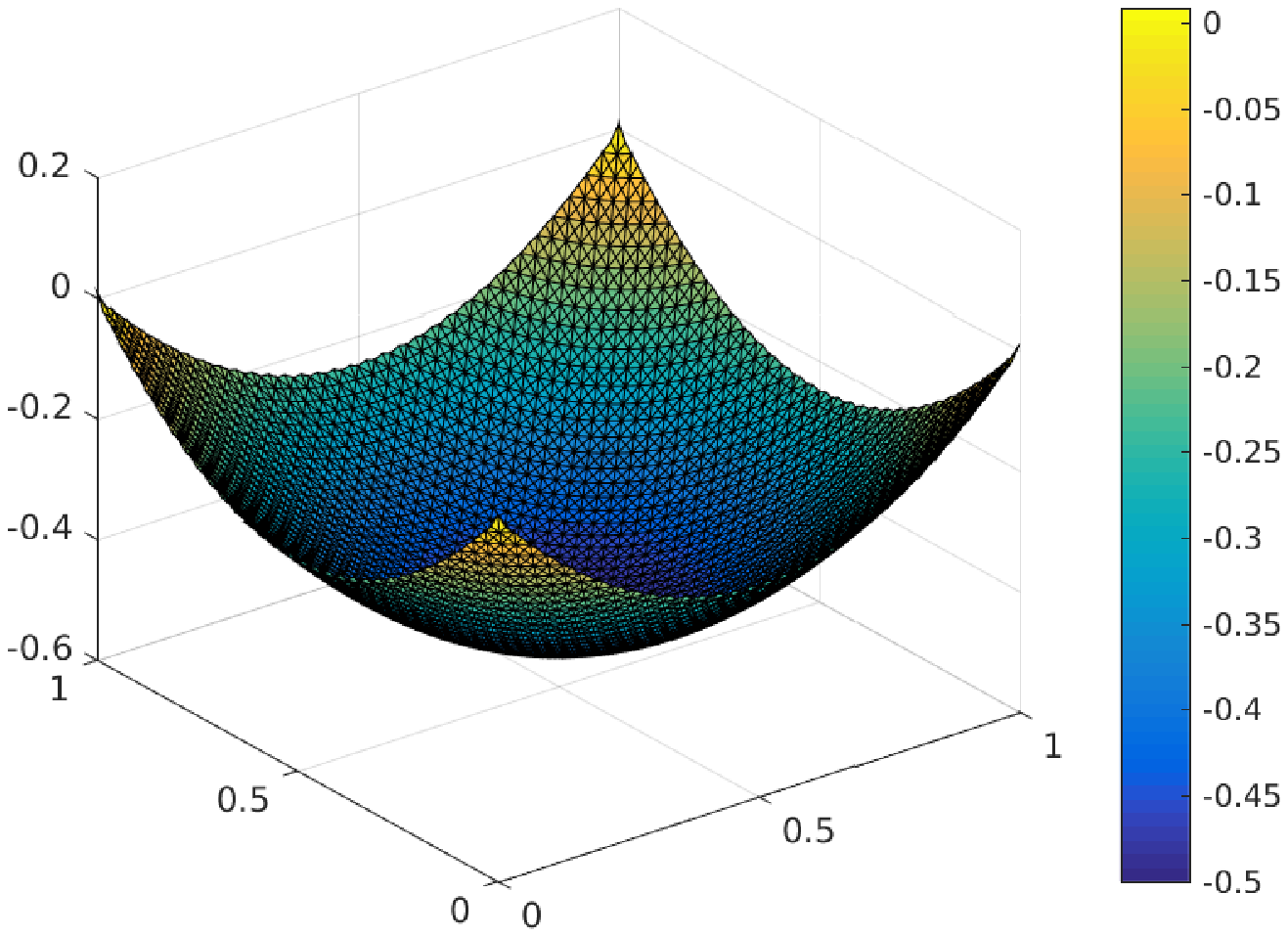}\\
      \end{minipage}
    \begin{minipage}[t]{0.5\linewidth}
        \centering
        \includegraphics[width=2.25in]{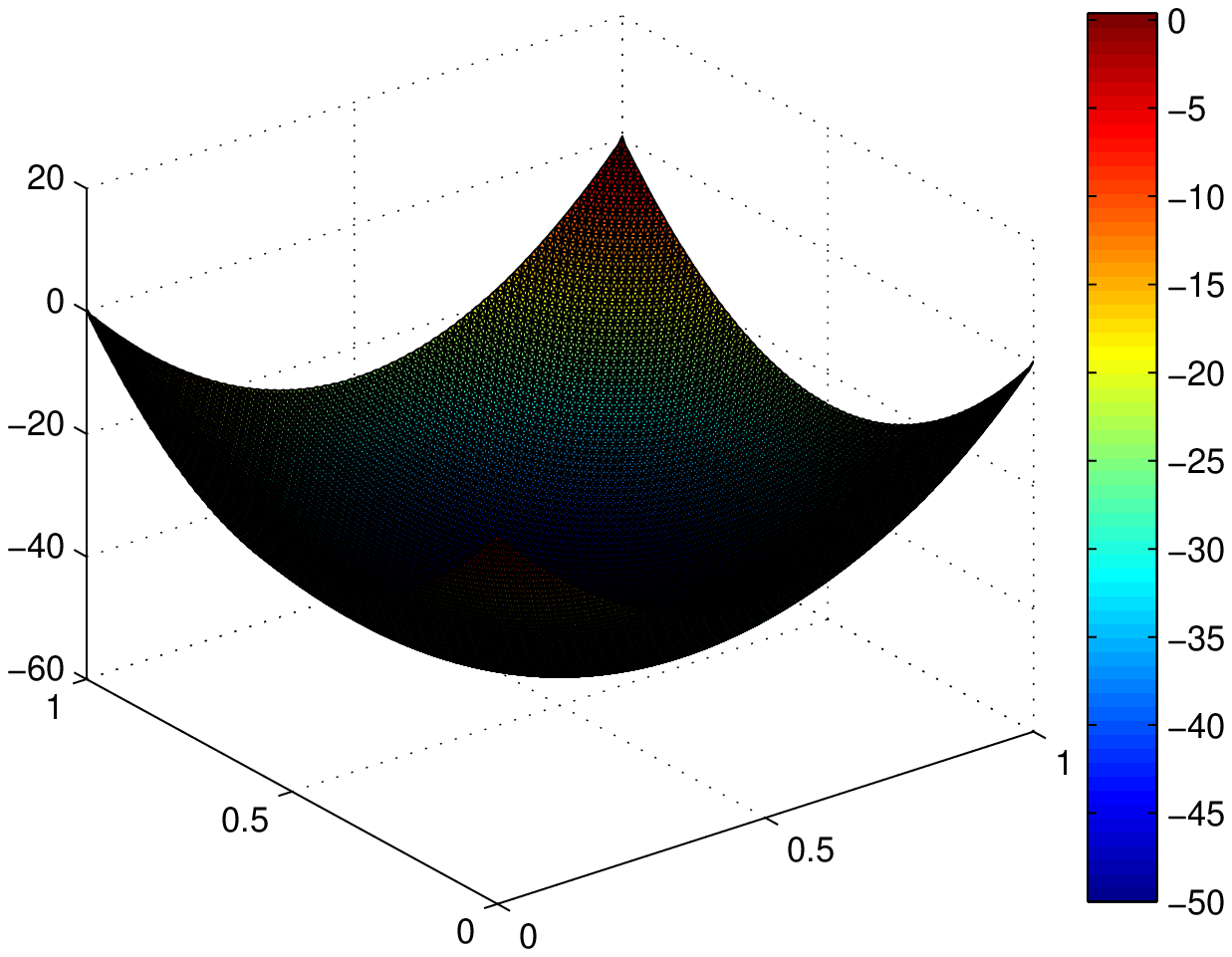}\\
    \end{minipage}
    \caption{Left: regularization parameter $\gamma=1$, an approximation to the
    state variable $y$ over the mesh with 8192 elements generated by uniform refinement of iterations 5. Right: regularization parameter $\gamma=0.01$, an approximation to the
    state variable $y$ over the mesh with 32768 elements generated by uniform refinement of iterations 6. }
    \label{du80}
\end{figure}

\begin{figure}[t]
    \begin{minipage}[t]{0.5\linewidth}
        \centering
        \includegraphics[width=2.25in]{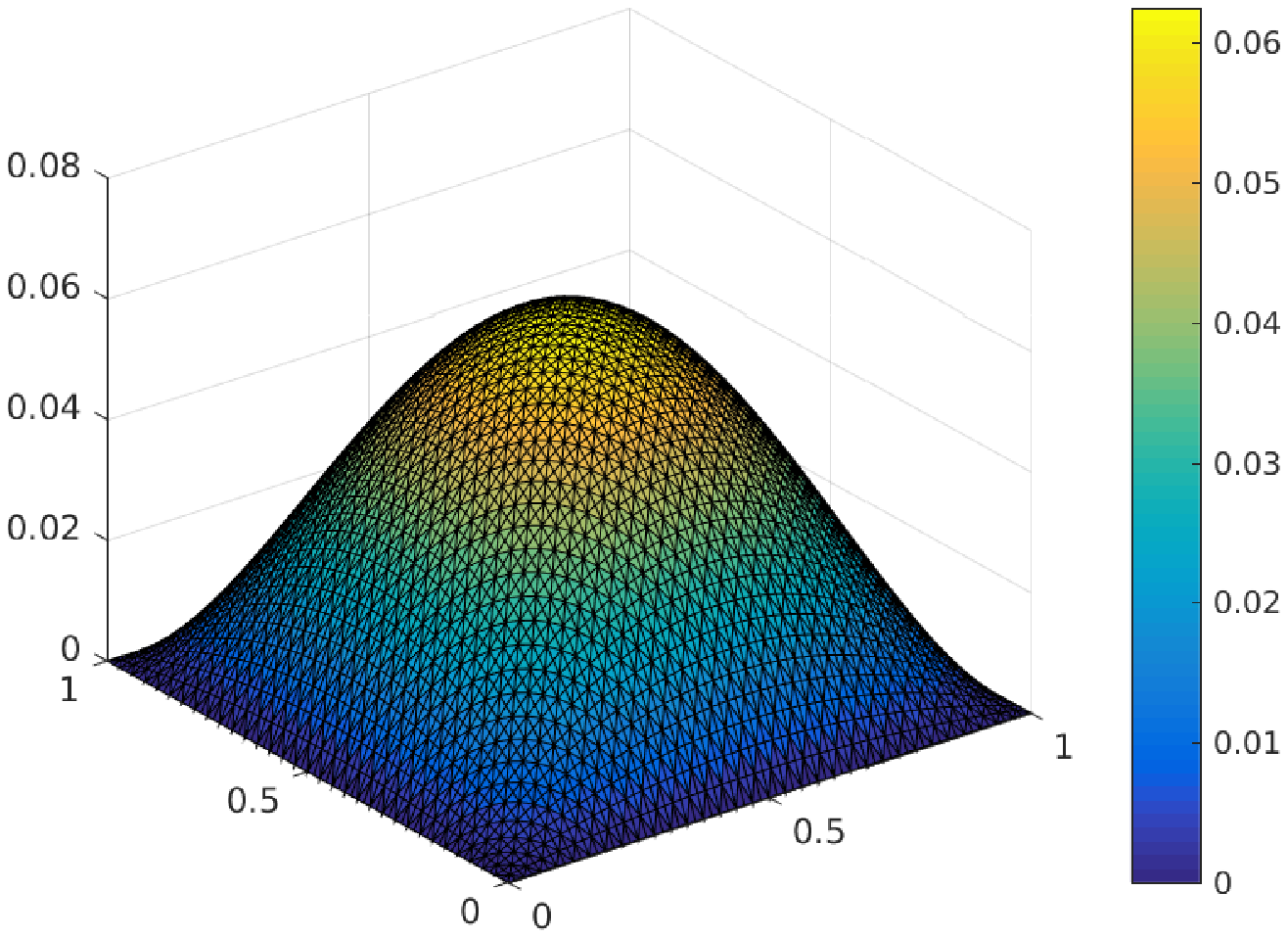}\\
      \end{minipage}
    \begin{minipage}[t]{0.5\linewidth}
        \centering
        \includegraphics[width=2.25in]{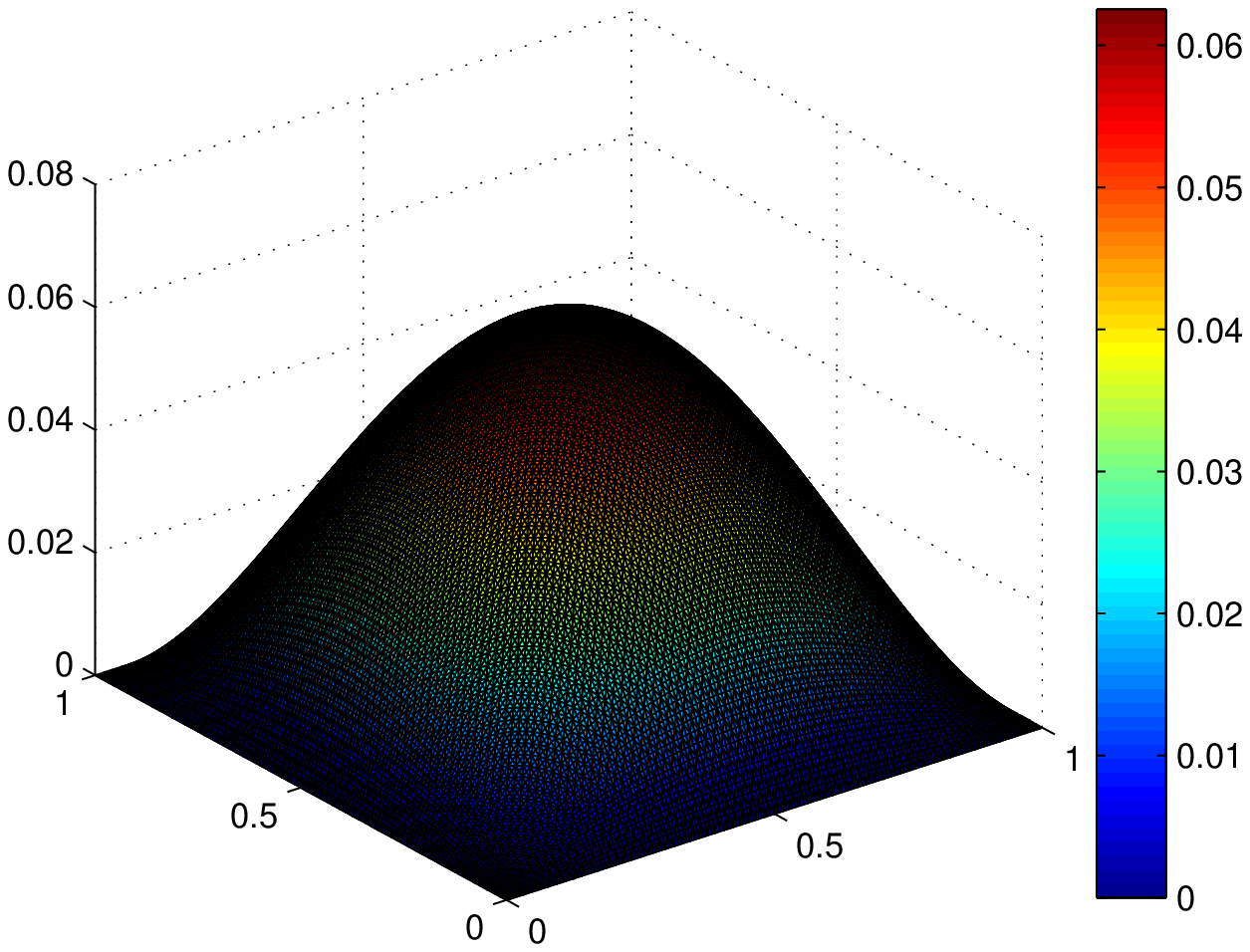}\\
    \end{minipage}
    \caption{An approximation to the adjoint state $z$ over the mesh with 8192 elements for $\gamma=1$ (left) and over the mesh
    with 32768 elements for $\gamma=0.01$ (right).}
    \label{du81}
\end{figure}

\begin{figure}[t]
    \begin{minipage}[t]{0.5\linewidth}
        \centering
        \includegraphics[width=2.25in]{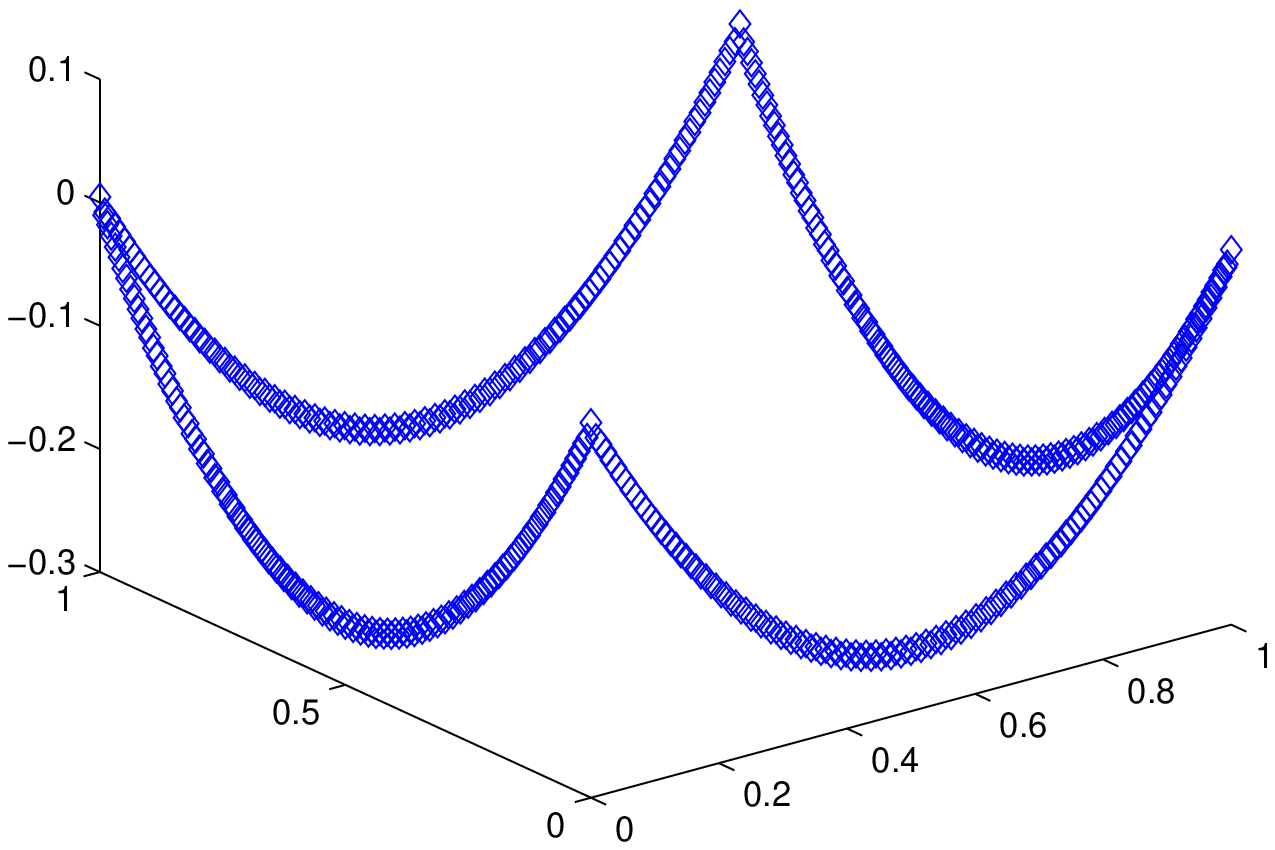}\\
      \end{minipage}
    \begin{minipage}[t]{0.5\linewidth}
        \centering
        \includegraphics[width=2.25in]{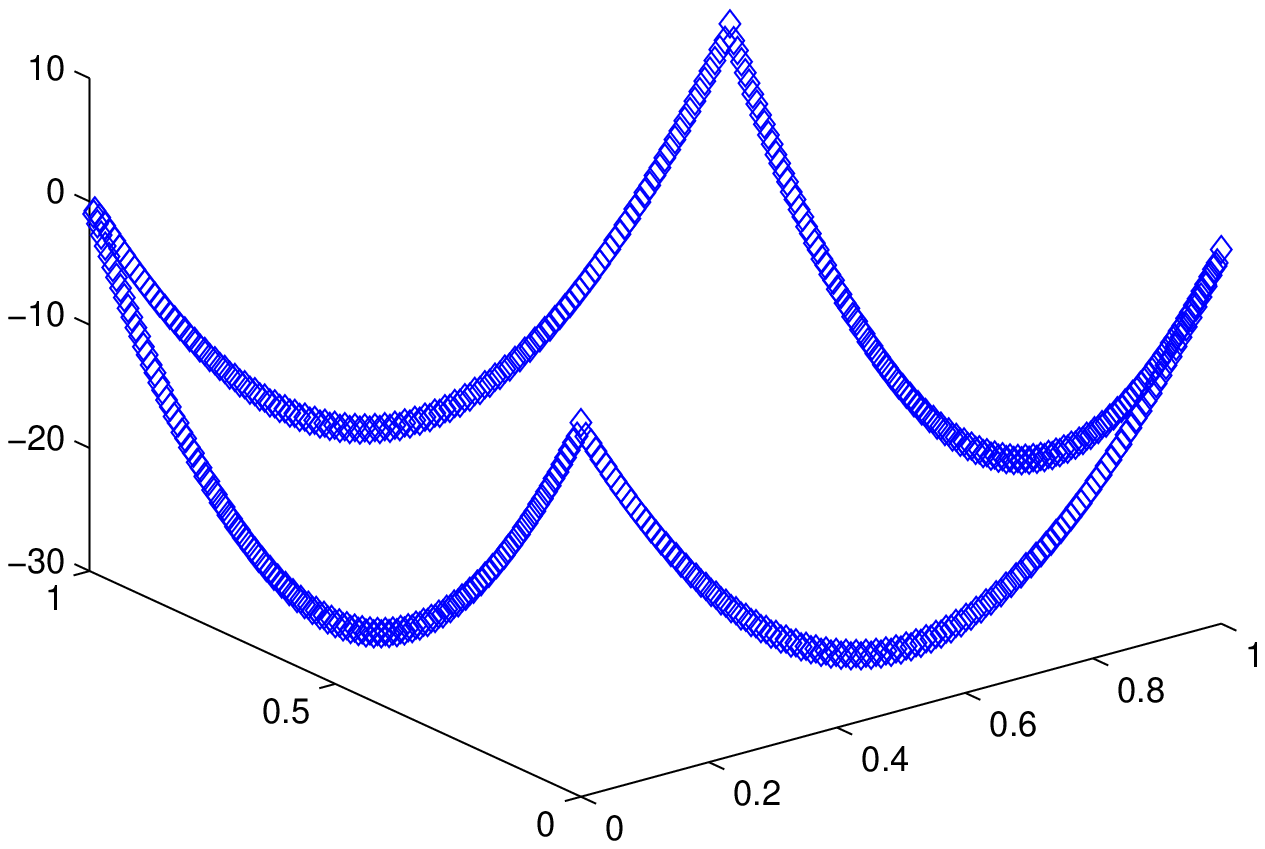}\\
    \end{minipage}
    \caption{An approximation to the
     control variable $u$, i.e., a restriction of $y_{h}$ to the boundary $\Gamma$, over the mesh with 32768 elements generated by uniform refinement of iterations 6 for regularization parameter $\gamma=1$ (left) and $\gamma=0.01$ (right).}
    \label{du48}
\end{figure}

We start with an initial mesh consisting of 8 congruent right triangles. Figure \ref{du80} reports an approximation solution of the state variable $y$ over the mesh with 8192 elements, which
generated by uniform refinement of iterations 5 for regularization parameter $\gamma=1$ (left), and over the mesh with 32768 elements, which
generated by uniform refinement of iterations 6 for regularization parameter $\gamma=0.01$ (right).
In Figure \ref{du81}, we depict the pictures of an approximation solution of the adjoint state $z$ over the mesh with 8192 elements for $\gamma=1$ (left) and over the mesh with 32768 elements for $\gamma=0.01$ (right). Figure \ref{du48} shows an restriction
(which is regarded as an approximation solution of the control variable $u$) of $y_{h}$ on the boundary over the mesh with 32768 elements for  $\gamma=1$ (left) and for $\gamma=0.01$ (right).

Table \ref{du78} shows respectively the exact errors $\|\nabla(y-y_{h})\|,\|\nabla(z-z_{h})\|$ and $\|u-u_{h}\|_{0,\Gamma}$ for
the regularization parameter $\gamma=1$. It is observed that they have the rate of convergence of order one for linear element, which is
order half higher than theoretical results. Table \ref{du79} reports the true errors of the state and adjoint state in $L^{2}$ norm
for $\gamma=1$. It can be seen that $||y-y_{h}||$ has the rate of convergence of order 1.5 at least, and that the speed of convergence of $||z-z_{h}||$ is close to 2. Table \ref{du47} provides the exact errors of $||y-y_{h}||,||\nabla(z-z_{h})||$
and $||u-u_{h}||_{0,\Gamma}$ for the regularization parameter $\gamma=0.01$, and the similar rate of convergence to $\gamma=1$ can be observed.

In addition, comparing Table \ref{du78} with Table \ref{du79}, we can see that the speed of convergence of $||y-y_{h}||$ is order half higher
than $||\nabla(y-y_{h})||$, and that the rate of convergence of $||z-z_{h}||$ is order one higher than $||\nabla(z-z_{h})||$

\begin{table}[t]\small
 \begin{center}
        \caption{Numerical data of $\gamma=1$ for Example 1: $h$ -- maximum size of quasi-uniform mesh; $||\nabla(y-y_{h})||$ -- numerical error for the state variable $y$; order$_{y}$ -- the speed of convergence for $y$; $||\nabla(z-z_{h})||$ -- numerical error for the adjoint state variable $z$; order$_{z}$ -- the speed of convergence for $z$; $||u-u_{h}||_{0,\Gamma}$ -- numerical error for the control variable $u$; order$_{u}$ -- the speed of convergence for $u$.} \label{du78}
        \small 
        \begin{tabular}{|c|c|c|c|c|c|c|} \hline
            $h$& $||\nabla(y-y_{h})||$& ${\rm order}_{y}$& $||\nabla(z-z_{h})||$& ${\rm order}_{z}$& $||u-u_{h}||_{0,\Gamma}$& ${\rm order}_{u}$\\ \hline
            0.7071&0.7187&--&0.1069&--&0.1901&-- \\ \hline
            0.3536&0.3603&0.9964&0.0539&0.9881&0.0663&1.5200\\ \hline
            0.1768&0.1928&0.9021&0.0278&0.9552&0.0345&0.9424\\ \hline
            0.0884&0.0898&1.1023&0.0140&0.9897&0.0154&1.1637\\ \hline
            0.0442&0.0446&1.0097&0.0070&1.000&0.0066&1.2224\\ \hline
        \end{tabular}
    \end{center}
\end{table}

\begin{table}[t]\small
 \begin{center}
        \caption{Numerical data of $\gamma=1$ for Example 1: $h$ -- maximum size of quasi-uniform mesh; $||y-y_{h}||$ -- numerical error for the state variable $y$; ${\rm order}_{y}$ -- the speed of convergence for $y$ in $L^{2}$ norm; $||z-z_{h}||$ -- numerical error for the adjoint state variable $z$; ${\rm order}_{z}$ -- the speed of convergence for $z$ in $L^{2}$ norm.}
        \label{du79}
        \small 
        \begin{tabular}{|c|c|c|c|c|c|c|} \hline
            $h$ &0.7071 & 0.3536& 0.1768& 0.0884& 0.0442 &0.0221\\ \hline
            $||y-y_{h}||$&0.0897&0.0250&0.0078&0.0025&7.86e-004&2.55e-004 \\ \hline
            ${\rm order}_{y}$& --&1.8436&1.6804&1.6415&1.6686&1.6259\\ \hline
            $||z-z_{h}||$&0.0181&0.0054&0.0014&3.55e-004&8.74e-005&2.10e-005\\ \hline
            ${\rm order}_{z}$& --&1.7453&1.9475&1.9775&2.0234&2.0604\\ \hline
       \end{tabular}
 \end{center}
\end{table}

\begin{table}[t]\small
 \begin{center}
        \caption{Numerical data of $\gamma=0.01$ for Example 1: $h$ -- maximum size of quasi-uniform mesh; $||y-y_{h}||$ -- numerical error for the state variable $y$; order$_{y}$ -- the speed of convergence for $y$ in $L^{2}$ norm; $||\nabla(z-z_{h})||$ -- numerical error for the adjoint state variable $z$; order$_{z}$ -- the speed of convergence for $z$; $||u-u_{h}||_{0,\Gamma}$ -- numerical error for the control variable $u$; order$_{u}$ -- the speed of convergence for $u$.} \label{du47}
        \small 
        \begin{tabular}{|c|c|c|c|c|c|c|} \hline
            $h$& $||y-y_{h}||$& ${\rm order}_{y}$& $||\nabla(z-z_{h})||$& ${\rm order}_{z}$& $||u-u_{h}||_{0,\Gamma}$& ${\rm order}_{u}$\\ \hline
            0.7071&2.9637&--&0.1134&--&9.0693&-- \\ \hline
            0.3536&0.7594&1.9649&0.0561&1.0156&2.5525&1.8295\\ \hline
            0.1768&0.2101&1.8538& 0.0279&1.0077&0.8519&1.5832\\ \hline
            0.0884&0.0663&1.6640&0.0140&0.9948&0.3384&1.3320\\ \hline
            0.0442&0.0191&1.7954&0.0070&1.000&0.1187&1.5114\\ \hline
        \end{tabular}
    \end{center}
\end{table}

\subsection{Example two}
We consider a 2D example over a square domain $\Omega=(0,1/4)\times(0,1/4)\subset\mathbb{R}^{2}$. The data is chosen as
\begin{equation*}
f=0, \ \ \ y_{d}=\left(x_{1}^2+x_{2}^2\right)^s,
\end{equation*}
where $s=10^{-5}$. Since we do not have an explicit expression for the exact solution, the ``reference solution" has been calculated
over a fine mesh with 131072 elements. Here, we also consider two settings, $\gamma=1$ and $\gamma=0.01$, of regularization parameter.

\begin{figure}[t]
    \begin{minipage}[t]{0.5\linewidth}
        \centering
        \includegraphics[width=2.25in]{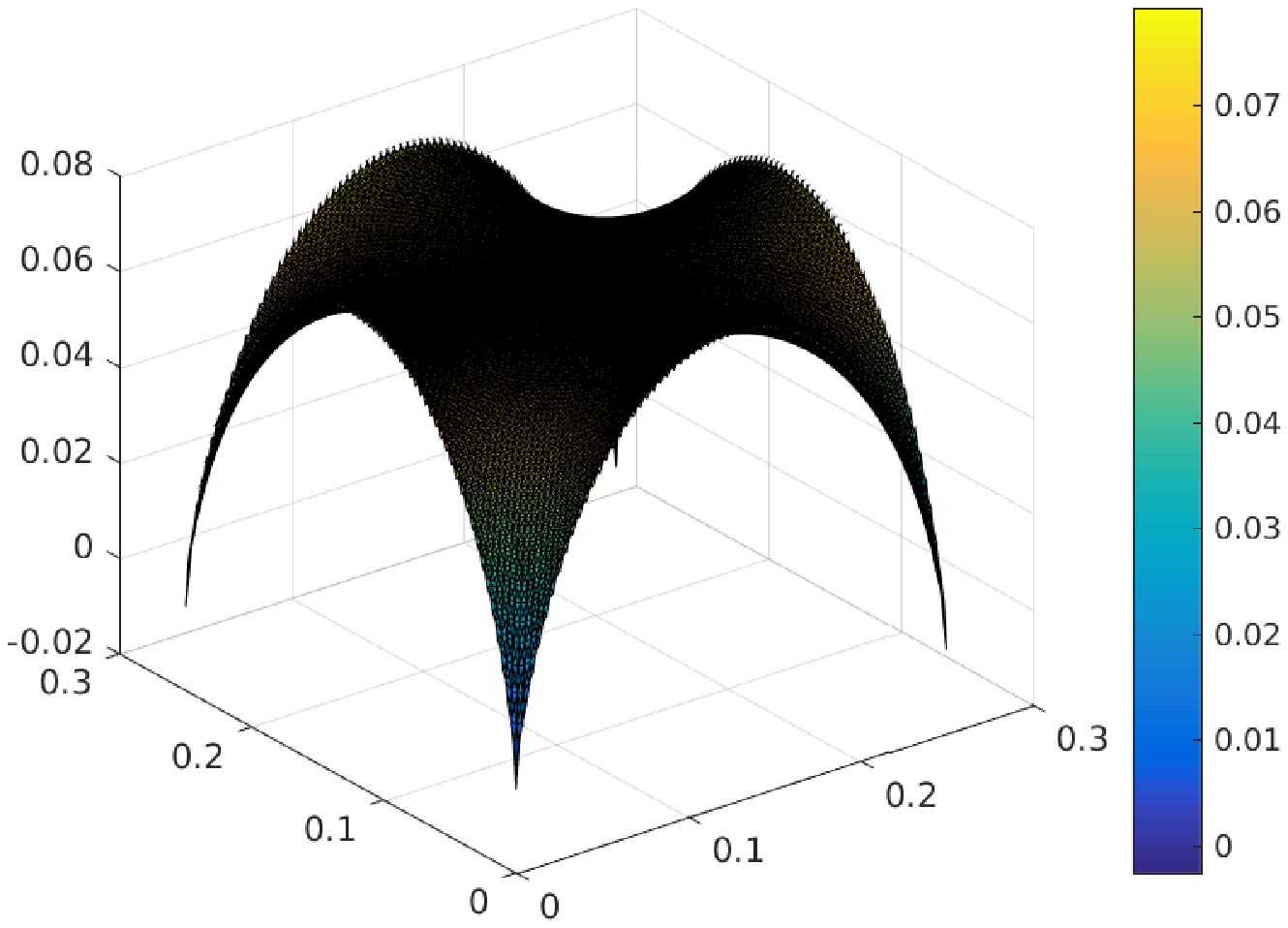}\\
      \end{minipage}
    \begin{minipage}[t]{0.5\linewidth}
        \centering
        \includegraphics[width=2.25in]{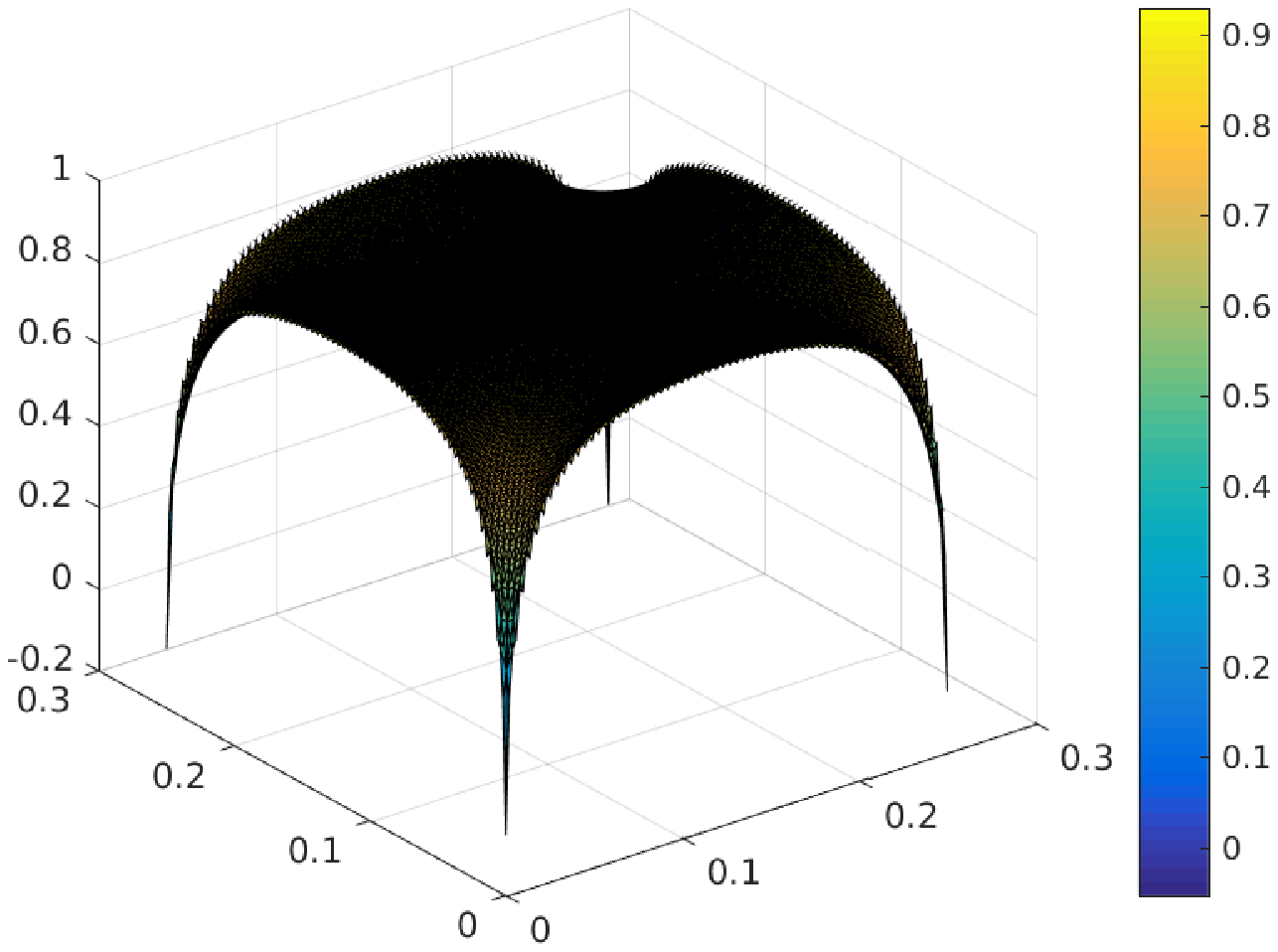}\\
    \end{minipage}
    \caption{An approximation solution to the state variable $y$ over the mesh generated by uniform refinement of iteration 6 (with 32768 elements) for the regularization parameter $\gamma=1$ (left) and $\gamma=0.01$ (right). }
    \label{du49}
\end{figure}
\begin{figure}[t]
    \begin{minipage}[t]{0.5\linewidth}
        \centering
        \includegraphics[width=2.25in]{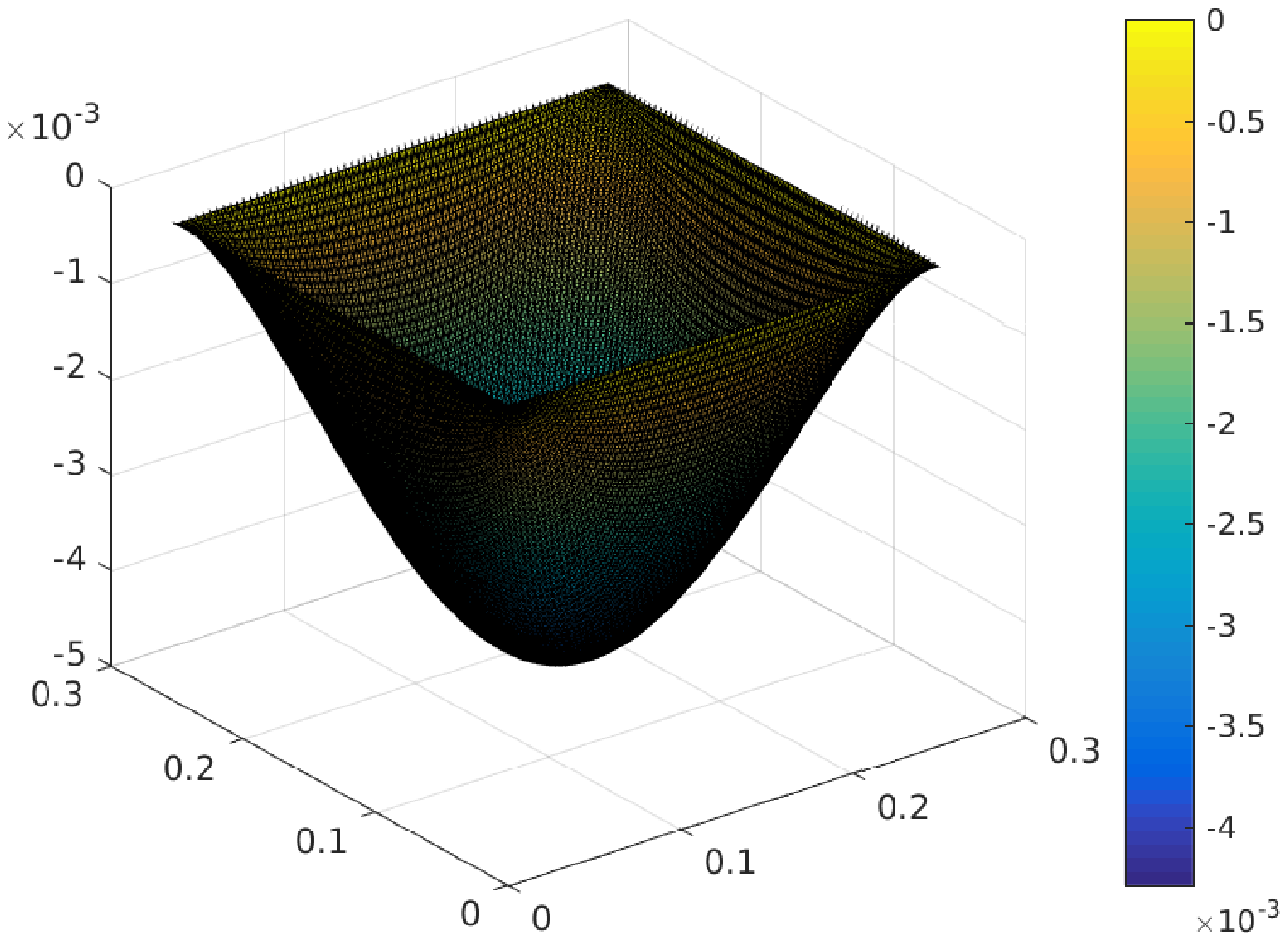}\\
      \end{minipage}
    \begin{minipage}[t]{0.5\linewidth}
        \centering
        \includegraphics[width=2.25in]{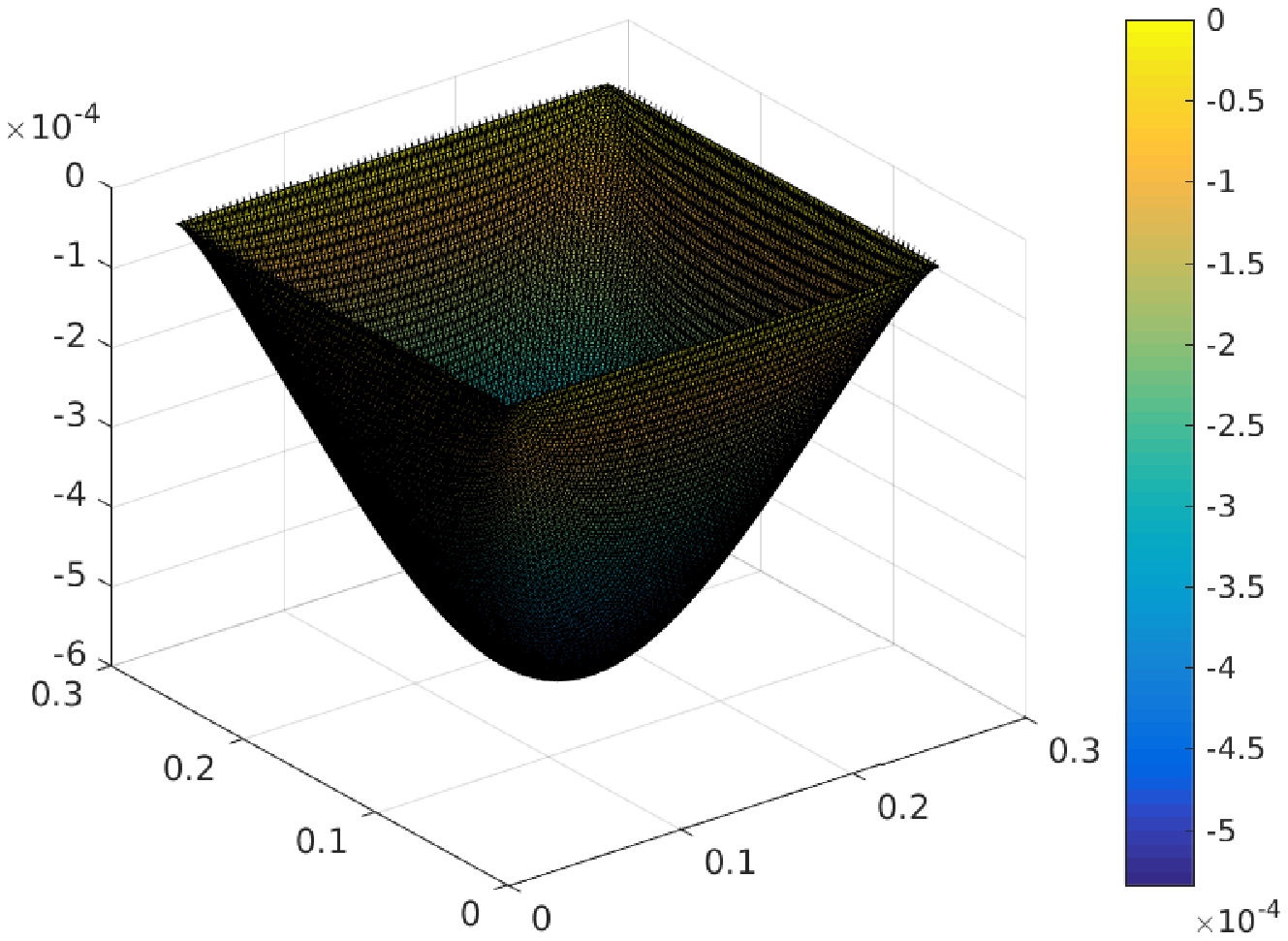}\\
    \end{minipage}
    \caption{An approximation solution to the adjoint state $z$ over the mesh generated by uniform refinement of iteration 6 (with 32768 elements) for the regularization parameter $\gamma=1$ (left) and $\gamma=0.01$ (right).}
    \label{du50}
\end{figure}

\begin{figure}[t]
    \begin{minipage}[t]{0.5\linewidth}
        \centering
        \includegraphics[width=2.25in]{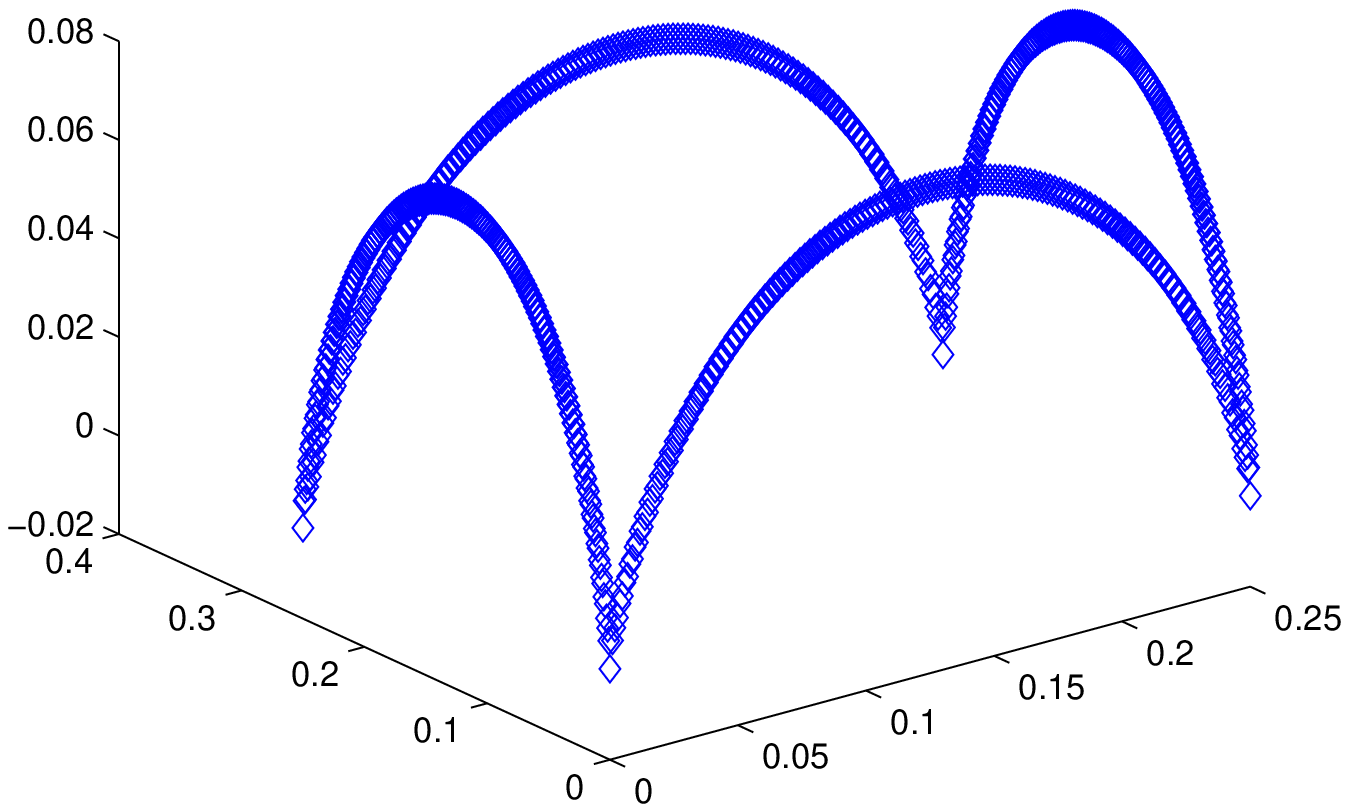}\\
      \end{minipage}
    \begin{minipage}[t]{0.5\linewidth}
        \centering
        \includegraphics[width=2.25in]{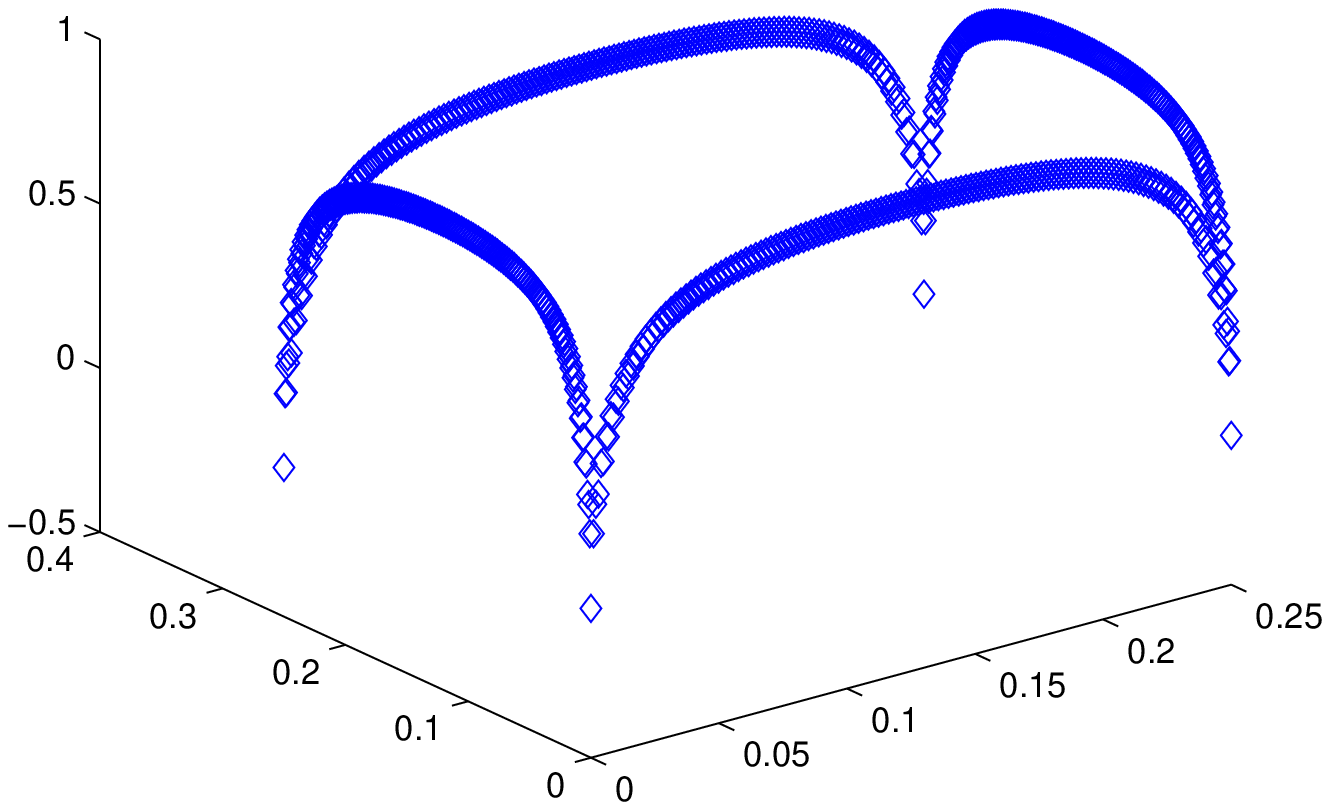}\\
    \end{minipage}
    \caption{An approximation to the
     control variable $u$, i.e., the restriction of $y_{h}$ on the boundary $\Gamma$, over the mesh generated by uniform refinement of iteration 7 (with 131072 elements) for the regularization parameter $\gamma=1$ (left) and $\gamma=0.01$ (right).}
    \label{du51}
\end{figure}

We still start with an initial mesh consisting of 8 congruent right triangles. Figures \ref{du49} and \ref{du50} show an approximation solution to the state $y$ and adjoint state $z$ over the mesh generated by uniform refinement of iteration 6 (with 32768 element) for different regularization parameter $\gamma=1$ (left) and $\gamma=0.01$ (right). Figure \ref{du51} reports the restriction of an approximation of the state on the boundary, i.e., an approximation solution of the control $u$, over the mesh generated by uniform refinement of iteration 7 (with 131072 element) for different regularization parameter $\gamma=1$ (left) and $\gamma=0.01$ (right).

From Figures \ref{du49} and \ref{du51}, we observe that the control changes quickly at the four corners of the boundary $\Gamma$. Furthermore, we remark that the control for the regularization parameter $\gamma=0.01$ changes more sharply at the four corners of the boundary than for $\gamma=1$, and that the singularity of the exact solution for $\gamma=0.01$ is stronger than for $\gamma=1$.

\begin{table}[t]\small
 \begin{center}
        \caption{Numerical data of $\gamma=1$ for Example 2: $h$ -- maximum size of quasi-uniform mesh; $||y-y_{h}||$ -- numerical error for the state variable $y$ in $L^{2}$ norm; order$_{y}$ -- the speed of convergence for $y$ ; $||z-z_{h}||$ -- numerical error for the adjoint state variable $z$ in $L^{2}$ norm; order$_{z}$ -- the speed of convergence for $z$; $||u-u_{h}||_{0,\Gamma}$ -- numerical error for the control variable $u$; $||\nabla(y-y_{h})||$ -- numerical error for the state variable $y$ in $H^{1}$ seminorm.} \label{du52}
        \small 
        \begin{tabular}{|c|c|c|c|c|c|c|} \hline
            $h$& $||y-y_{h}||$& ${\rm order}_{y}$& $||z-z_{h}||$& ${\rm order}_{z}$& $||u-u_{h}||_{0,\Gamma}$& $||\nabla(y-y_{h}||)$\\ \hline
            0.1768&0.0117&--&7.28e-005&--&0.3212&0.4462 \\ \hline
            0.0884&0.0034&1.7833&2.51e-005&1.5398&0.2999&0.3283\\ \hline
            0.0442&0.0011&1.6280&7.95e-006&1.6554&0.2896&0.3167\\ \hline
            0.0221&3.61e-004&1.6078&2.10e-006&1.9218&0.2794&0.3048\\ \hline
            0.0111&1.18e-004&1.6185&5.29e-007&1.9875&0.2693&0.2977\\ \hline
        \end{tabular}
    \end{center}
\end{table}

From Table \ref{du52}, we observe that the numerical error of the state $y$ in $L^{2}$ norm has the speed of convergence of order 1.6, and
that the rate of convergence of the numerical error for the adjoint state $z$ is still close to order 2. However, the numerical errors
$||u-u||_{0,\Gamma}$ and $||\nabla(y-y_{h})||$ have a very slow speed of convergence, this is due to the very low regularity of the
exact solutions. In fact, the exact control $u$ has strong singularity at four corners of the boundary. This indicates that adaptive mesh
 based on {a posteriori} error estimator is efficient to this type of problems, we refer to the articles \cite{cai1,Boundary,Ani,Babuska1978,NUMER,Verfurth1996,Liu2008,Li2002,Kohls2012,Kohls2014,Schneider2016,Becker2000,Hint2008} about adaptive finite element methods on the base of {a posteriori} error estimates.

\end{document}